\documentclass[preprint,12pt]{elsarticle}

\usepackage{amsmath,amssymb,amsfonts,amsthm,enumerate,multirow}
\usepackage{tikz}

\newtheorem{lemma}{Lemma}
\newtheorem{theorem}{Theorem}
\newtheorem{remark}{Remark}
\newtheorem{problem}{Problem}

\journal{SIAM Journal on Imaging Science}

\begin{document}

\begin{frontmatter}

\title{General reconstruction of elastic strain fields from their Longitudinal Ray Transform}

\author[Ncle]{CM Wensrich}
\author[Mchr]{S Holman}
\author[Mchr]{WRB Lionheart}
\author[Chile]{M Courdurier}
\author[Sobo]{A Polyakova}
\author[NSU]{I Svetov}
\author[Ncle]{T Doubikin}

\affiliation[Ncle]{organization={School of Engineering, University of Newcastle, Australia},%Department and Organization
            addressline={University Drive}, 
            city={Callaghan},
            postcode={2308}, 
            state={NSW},
            country={Australia}}

\affiliation[Mchr]{organization={School of Mathematics, University of Manchester},%Department and Organization
            addressline={Alan Turing Building, Oxford Rd}, 
            city={Manchester},
            postcode={M13 9PL}, 
            country={UK}}

\affiliation[Chile]{organization={Department of Mathematics, Pontificia Universidad Católica de Chile},%Department and Organization
            addressline={Avda. Vicuña Mackenna 4860}, 
            city={Macul, Santiago},
            country={Chile}}

\affiliation[Sobo]{organization={Sobolev Institute of Mathematics},%Department and Organization
            addressline={ 630090}, 
            city={Novosibirsk},
            country={Russia}}

\affiliation[NSU]{organization={Novosibirsk State University},%Department and Organization
            addressline={ 630090}, 
            city={Novosibirsk},
            country={Russia}}

\begin{abstract}

We develop an algorithm for reconstruction of elastic strain fields from their Longitudinal Ray Transform (LRT) in either two or three dimensions. In general, the LRT only determines the solenoidal part of a symmetric tensor field, but elastic strain fields additionally satisfy mechanical equilibrium, an extra condition that allows for full reconstruction 
in many cases. Our method provides full reconstruction for general elastic strain fields in connected objects whose boundary only contains one component, while previous results included other requirements such as no residual stress, or zero boundary traction. This allows for full reconstruction in energy resolved neutron transmission imaging for simple objects. Along the way, we prove that the LRT of a potential rank-2  tensor restricted to a bounded set determines the potential on the boundary of the set up to infinitesimal rigid motions on each component of the boundary. The method is demonstrated with numerical examples in two dimensions.  

\end{abstract}

%%%Graphical abstract
%\begin{graphicalabstract}
%%\includegraphics{grabs}
%\end{graphicalabstract}

%%%Research highlights
%\begin{highlights}
%\item Research highlight 1
%\item Research highlight 2
%\end{highlights}

\begin{keyword}
%% keywords here, in the form: keyword \sep keyword
Strain tomography \sep Longitudinal Ray Transform \sep Bragg edge \sep Neutron transmission

%% PACS codes here, in the form: \PACS code \sep code

%% MSC codes here, in the form: \MSC code \sep code
%% or \MSC[2008] code \sep code (2000 is the default)

\end{keyword}

\end{frontmatter}

%% \linenumbers

\section{Introduction and background}

Energy resolved neutron transmission imaging \cite{kiyanagi2012new,kockelmann2007energy,shinohara2020energy} poses a range of rich-tomography problems where the aim is to reconstruct non-scalar spatially distributed information from projected images (potentially non-scalar).  Bragg-edge strain tomography is a prominent example where the aim is to reconstruct the triaxial elastic strain field within a sample (a rank-2 tensor) from a set of scalar neutron transmission-based strain images \cite{abbey2012neutron,gregg2017tomographic,gregg2018tomographic,hendriks2019robust,hendriks2017bragg,santisteban2001time}.

With reference to Figure \ref{fig:BE_Imaging}, this technique revolves around measuring inter-atomic lattice spacings within a sample from abrupt changes in neutron transmission rate as a function of wavelength.  Known as `Bragg-edges', these changes are due to coherent elastic scattering and occur at well-defined wavelengths related to the crystal structure and lattice spacing.  Relative change in the location of these edges provides a direct measurement of elastic strain within a sample of interest.  A detailed overview of the technique and its application can be found in the literature (e.g. see \cite{santisteban2001time}), however the salient features from our point of view is that the measured value of strain at each pixel is the longitudinal projection of elastic strain, $\epsilon$, averaged along the path of a corresponding ray of the form
\begin{equation}
        \frac{1}{L} \int_{0}^L \epsilon_{ij}(x_0+t\xi)\xi_{i}\xi_{j}dt,
        \label{BEStrain}
\end{equation}
where $x_0$ is the point on the surface where the ray enters the sample, $L$ is the projected thickness of the sample (as observed by the ray) and $\xi$ is a unit vector indicating the direction of the ray (see Figure \ref{LRTGeom}).  Strain measurements of this type exclusively refer to the elastic component of strain; the measurement is insensitive to inelastic strains due to plasticity, mechanical interference, etc.

\begin{figure}[tb]
        \vspace{-1ex}
    	\centering
        \includegraphics[width=0.95\linewidth]{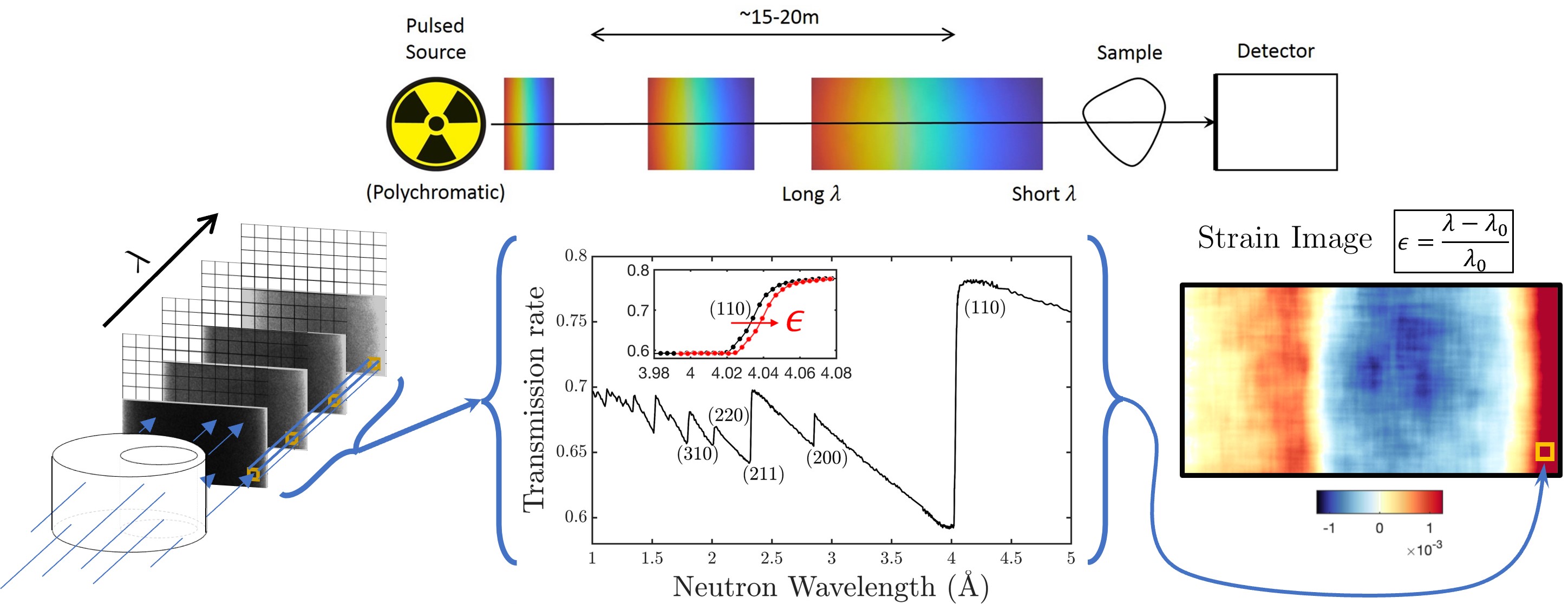}
  	\put(-300,120){(a)}
	  \put(-360,85){(b)}
        \put(-270,80){(c)}
	  \put(-105,85){(d)}
    	\caption{Bragg-edge imaging of strain at a pulsed neutron source.  (a) Pulses of neutrons are generated at a source (usually a spallation target) before travelling a fixed distance down a guide to a sample and detector. High-energy (short wavelength) neutrons travel faster and arrive before low-energy (long wavelength) neutrons; the time-of-flight of a detected neutron is proportional to its wavelength. (b) Pixelated time-of-flight detectors now allow for \emph{rich}-imaging where a `radiograph' consists of a stack of images - one for each neutron wavelength in a spectrum.  (c) At each pixel, Bragg-edges are formed at wavelengths defined by the spacings of various crystal planes within the sample.  Relative change in the location of these edges provide a measure of strain averaged along the ray paths.  e.g. (d) A projected strain image from a steel `ring-and-plug' reference sample computed from relative shift in the (110) edge position. (adapted from data measured by Gregg \textit{et.al.} \cite{gregg2018tomographic}) \vspace{-2ex}}
  	\label{fig:BE_Imaging}
\end{figure}

\begin{figure}
\begin{center}
    \includegraphics[width=0.7\linewidth]{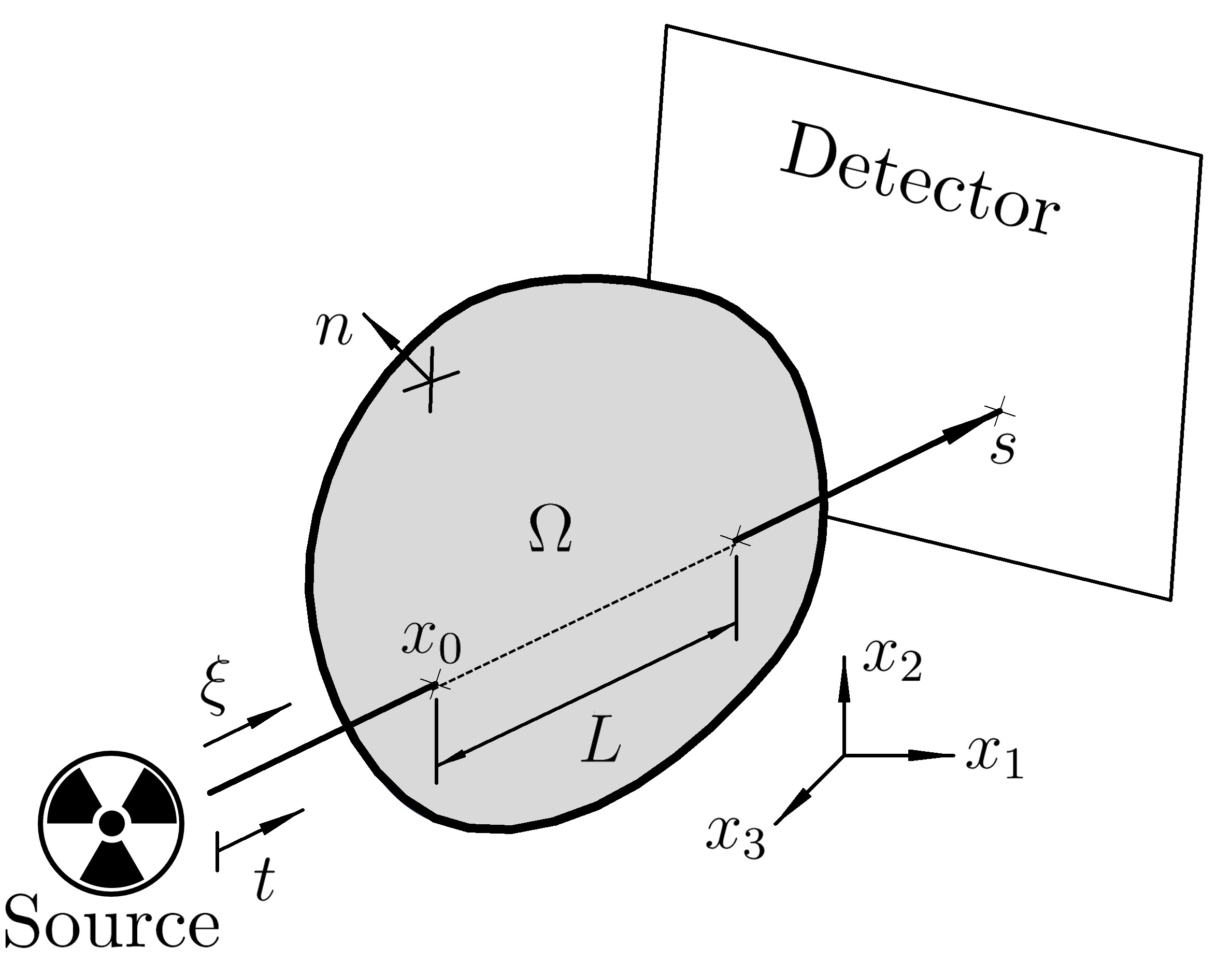}
    \caption{\label{LRTGeom}Geometry of a ray passing from a source, through the sample $\Omega$, to a detector.}
\end{center}
\end{figure} 

This physical measurement has a natural correspondence to the Longitudinal Ray Transform (LRT), $I$, which can be written for $f \in C^\infty_c(\mathcal{S}^2;\mathbb{R}^3)$ as
\begin{equation}
If(s,\xi)=\int_{-\infty}^\infty f_{ij}(s+t\xi)\xi_{i}\xi_{j} dt.
\end{equation}
Note that the extension of $I$ to all $f \in L^2(\mathcal{S}^2;\mathbb{R}^3)$ can be carried out in the usual way (see below for definitions and notation).  From this perspective, the practical problem becomes one of reconstructing the strain field within a sample from a set of measurements of its LRT.

In a recent publication \cite{wensrich2024direct}, we examine this problem in the context of residual strain fields in $\mathbb{R}^2$ for samples satisfying mechanical equilibrium and subject to zero boundary traction (i.e. no applied load).  In such case, we were able to construct and demonstrate a direct reconstruction technique based on an inversion formula together with a relationship between Airy stress fields and Helmholtz tensor decomposition for these systems.  In this present paper, we extend such approach to a general version in three-dimensions where samples may be subject to external loads in addition to internal residual stress.  

We begin by defining our notation and the spaces we will use before providing a brief review of the LRT and its inversion formulas.  We follow this with a discussion of the key points of our prior work and its adaptation to a general reconstruction algorithm for $\mathbb{R}^2$ and $\mathbb{R}^3$.  This algorithm is then demonstrated numerically on an example in $\mathbb{R}^2$.

\section{Notation and function spaces}

We largely adopt the same notation as our previous paper \cite{wensrich2024direct} and work in the following spaces, with $\Theta$ a domain, with Lipschitz boundary if bounded, in $\mathbb{R}^n$ where $n=2,3$:

\begin{itemize}
    \item[]{$L^2(\mathcal{S}^m;\Theta)$ -- The space of square-integrable rank-$m$ symmetric tensor fields on $\Theta$.  We are particularly interested in the case of $m=2$;}
	
    \item[] $H^1(\mathcal{S}^m;\Theta)$ -- The subspace of $L^2(\mathcal{S}^m;\Theta)$ of fields whose weak first derivatives are all square-integrable;
  
    \item[]{$\mathcal{C}^\infty(\mathcal{S}^m;\Theta)$ -- The space of smooth rank-$m$ symmetric tensor fields on $\Theta$ with continuous derivatives of all orders.}
\end{itemize}

We make use of the following differential operators:
\begin{itemize}
    \item[]{$d$ -- The symmetric gradient operator. For $f \in \mathcal{C}^\infty(\mathcal{S}^m;\Theta)$, $df \in \mathcal{C}^\infty(\mathcal{S}^{m+1};\Theta)$ will be the symmetric derivative defined in \cite{sharafutdinov2012integral}. This coincides with the gradient when $m = 0$ and for $u \in \mathcal{C}^\infty(\mathcal{S}^1;\Theta)$
    \[
    [du]_{ij} = \frac{1}{2}\left (\frac{\partial u_{i}}{\partial x_{j}} + \frac{\partial u_{j}}{\partial x_{i}} \right ),
    \]
    or equivalently $du=\tfrac{1}{2}\big(\nabla \otimes u + (\nabla \otimes u)^T\big)$, where $\otimes$ is the dyadic product;}
   
    \item[]
    {$\text{Div}$ -- The divergence operator which is the formal adjoint of $-d$ and maps $\mathcal{C}^\infty(\mathcal{S}^{m+1};\Theta) \rightarrow \mathcal{C}^\infty(\mathcal{S}^{m};\Theta)$. This is the contraction of the gradient of a tensor field and for the general formula see \cite{sharafutdinov2012integral}. 
    For $u \in \mathcal{C}^\infty(\mathcal{S}^1;\Theta)$, $\text{Div}(u)$ is the standard divergence of $u$.  For $f \in \mathcal{C}^\infty(\mathcal{S}^2;\Theta)$, \[
    [\text{Div}(f)]_i = \frac{\partial f_{ij}}{\partial x^j};
    \]
    }

    \item[]{$\Delta$ -- The Laplacian defined by $\Delta = \frac{\partial^2}{\partial x_k \partial x_k}$.  Fractional powers of $\Delta$, such as that found in \eqref{fullSharafutdinov},\eqref{2DInvFormula} and, \eqref{3DInvFormula}, are defined via the Fourier transform;}

    \item[]{$W$ -- The Saint Venant operator.  For $f\in\mathcal{C}^\infty(\mathcal{S}^2;\Theta)$, $Wf$ is a rank-4 tensor with components
\[
[Wf]_{ijkl}=\frac{\partial^2f_{ij}}{\partial x_k\partial x_l}  + \frac{\partial^2f_{kl}}{\partial x_i\partial x_j}  -\frac{\partial^2f_{il}}{\partial x_j\partial x_k} - \frac{\partial^2f_{jk}}{\partial x_i\partial x_l}.
\]
For $n=3$, $Wf$ has six unique components that can also be specified by the rank-2 symmetric incompatibility tensor $Rf=\nabla \times (\nabla \times f)^T$, or component-wise $[Rf]_{ij}=e_{kpi}e_{lqj}\nabla_p\nabla_q f_{kl}$ where $e_{ijk}$ is the Levi-Civita permutation symbol.  $R$ is identical to the $\textnormal{\bf CURLCURL}$ operator defined in \cite{CRMATH_2006}.
}
    
\end{itemize}

Additionally, we say that a tensor field is divergence-free if its divergence is zero. The differential operators are initially defined on smooth tensor fields, but can be extended to fields with distributional coefficients in the usual way.

We will mostly be concerned with tensors of rank either $m=1$ or $2$ and use the standard notations $f:g$ for contraction of rank-2 tensors and $f\cdot g$ for multiplication of a rank-2 tensor with a rank-1 tensor, or the dot product of rank-1 tensors.  When referring to components, we use standard summation convention in the case of repeated indices.

\section{LRT inversion formulas and Helmholtz decomposition}

As discussed at length by others (e.g. \cite{sharafutdinov2012integral}), the LRT has a large null space consisting of `potential' fields of the form $du$ for any $u$ that vanishes at infinity.  This provides a natural Helmholtz decomposition of $f$ into solenoidal (divergence-free) and potential parts of the form
\begin{equation}
\label{Helholtz}
    f = {^s}f + du,
\end{equation}
where $If=I\,^sf$.

The solenoidal part can be recovered from the LRT using an inversion formula due to Sharafutdinov \cite{sharafutdinov2012integral} which can be written in a general form for $f \in L^2(\mathcal{S}^m;\mathbb{R}^n)$ as
\begin{equation}
\label{fullSharafutdinov}
^sf=(-\Delta)^{1/2}\Big[\sum_{k=0}^{[m/2]}c_k(i-\Delta^{-1}d^2)^kj^k\Big]\mu^m If,
\end{equation}
where $c_k$ are defined scalar coefficients, the operators $i$ and $j$ respectively refer to product and contraction with the Kronecker tensor, and $\mu^m$ is the formal adjoint of $I$ when the measure on $\mathbb{S}^{n-1}$ is normalised to one.  In practical terms, $\mu^m$ is related to the adjoint of the X-ray transform (i.e. scalar back-projection)
\[
\mathcal{X}^*g(x)=\int_{\mathbb{S}^{n-1}} g(x,\xi) d\xi,
\]
where $d\xi$ is the surface measure on $\mathbb{S}^{n-1}$.  With $\mathcal{X}^*$ acting component-wise and each back-projection weighted by the diadic product of $\xi$ with itself $m$-times, $\mu^m$ can be written
\begin{equation}
\mu^m_{i_1i_2...i_m}= \frac{1}{2 \pi^{n/2}}\Gamma\left (\tfrac{n}{2} \right )\mathcal{X}^*\xi_{i_1}\xi_{i_2}...\xi_{i_m},
\end{equation}
where the constant factor is present because of the normalisation of the measure on $\mathbb{S}^{n-1}$ in \cite{sharafutdinov2012integral}. 

For rank-2 strain $\epsilon \in L^2(\mathcal{S}^2;\mathbb{R}^2)$, \eqref{fullSharafutdinov} can be written
\begin{align}
\label{2DInvFormula}
^s\epsilon&=\frac{1}{4 \pi} (-\Delta)^{1/2}\mathcal{X}^*\xi\otimes\xi I\epsilon\\
          &=\frac{1}{4 \pi} \mathcal{X}^* \Lambda \xi\otimes\xi I\epsilon,
\end{align}
where $\Lambda$ is a ramp filter as used in standard Filtered Back Projection \cite{wensrich2024direct,derevtsov2015tomography,louis2022inversion}.  For $\epsilon \in L^2(\mathcal{S}^2;\mathbb{R}^3)$
\begin{equation}
{^s}\epsilon = \frac{1}{4\pi^2}(- \Delta)^{1/2}\Big[4 - (\text{\bf{I}}-\Delta^{-1}d^2) \textnormal{tr} \Big] \mathcal{X}^* \xi\otimes\xi I\epsilon,
\label{3DInvFormula}
\end{equation}
where $\text{\bf{I}}$ is the rank-2 identity, and $\textnormal{tr}$ is the trace operator.

Our approach to the recovery of $\epsilon$ from $I\epsilon$ is to use these inversion formulas together with physical laws to render the problem well-posed.  This process is outlined in the following section.

\section{Reconstruction in the presence of boundary traction} \label{sec:bt}

Consider a bounded elastic body with Lipschitz boundary $\Omega \subset \mathbb{R}^n$ for $n=2,3$ in the absence of body forces (e.g. gravity, magnetic forces, etc.).  Mechanical equilibrium implies that stress, $\sigma$, at every point within $\Omega$ is divergence-free which, when combined with Hooke's law, holds that
\begin{equation}
    \label{Equilibrium}
    \text{Div}(C:\epsilon)=0,
\end{equation}
where $C$ is the usual 4-rank tensor of elastic constants ($\sigma_{ij}=C_{ijkl}\epsilon_{kl}$).

From this perspective we formulate the following inverse problem that we seek to solve;
\begin{problem}
    \label{ourproblem}
    Let $\Omega \subset \mathbb{R}^n$ where $n=2,3$ be a bounded domain with Lipschitz boundary. Consider $\epsilon \in L^2(\mathcal{S}^2;\Omega)$ such that ${\rm Div}(C:\epsilon)=0$ where $C$ is a positive definite rank-4 tensor of elastic constants.  We wish to recover $\epsilon$ from its LRT $I\epsilon$.
\end{problem}

We now consider $\sigma$ and $\epsilon$ to be extended by zero outside $\Omega$ and observe that \eqref{Equilibrium} holds everywhere except on $\partial\Omega$.  The key observation in our previous paper was that a zero-traction boundary condition is equivalent to \eqref{Equilibrium} being satisfied on the boundary (at least in a distributional sense).  This allowed us to consider stress as divergence-free at all points in the domain, and in this case for $n=2$, we were able to;
\begin{enumerate}
    \item{Show that $\text{supp}({^s}\epsilon)\subset\Omega$}, and,
    \item{Applying elastic properties, express $\sigma$ in terms of $^{s}\epsilon$.}
\end{enumerate}
These results provided two separate avenues to fully recover $\epsilon$ in this special case.  In particular, one of these avenues involved calculating the potential part of \eqref{Helholtz} as the solution to an elliptic differential equation following from \eqref{Equilibrium} along with with the boundary condition $u|_{\partial \Omega}=0$.

We now consider the more general case of samples in $n=2,3$ dimensions that may be subject to external loads (i.e. non-zero boundary traction) in addition to residual stress (i.e. stress due to inelastic mechanisms such as plastic deformation and phase changes).  In this setting $\text{supp}(\epsilon)$ is still restricted to $\Omega$ but $\text{supp}({^s}\epsilon)$ and $\text{supp}(du)$ may both be unbounded. 

However, introducing the characteristic function
\[
\chi(x)=\begin{cases}
    1 & \text{for $x\in\Omega$}\\
    0 & \text{otherwise}
  \end{cases}
\]
we can write
\begin{align*}
    I\epsilon &= I\chi\epsilon\\
            &= I\chi\,{^s}\epsilon+I\chi du.
\end{align*}

Hence it is possible to write the LRT of the potential part restricted to $\Omega$ as
\begin{equation}
\label{restrictedLRT}
    I\chi du = I\epsilon - I\chi \, {^s}\epsilon
\end{equation}
Consider the ray represented by the line $t \mapsto s + t \xi$ which passes though $\Omega$ intersecting its boundary transversely at a sequence of points $\{t_j\}_{j=1}^{2N}$ (see Figure \ref{RayGeom}). 
\begin{figure}
\begin{center}
    \includegraphics[width=0.4\linewidth]{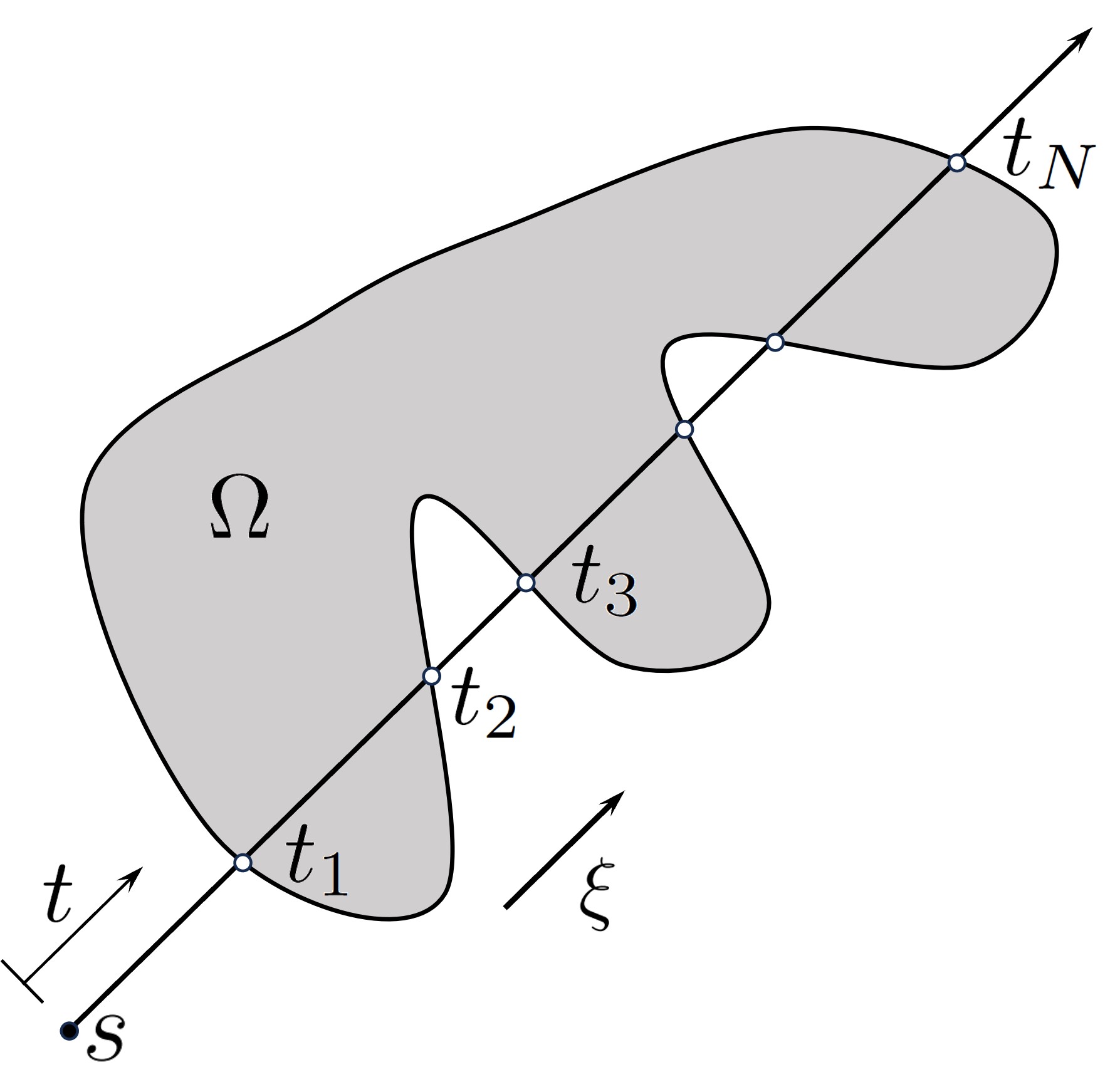}
    \caption{\label{RayGeom}Geometry of a ray passing through a sample $\Omega$ and intersecting the boundary at multiple points.}
\end{center}
\end{figure} 
For this ray we can write \eqref{restrictedLRT} as
\[
\sum_{j =1}^N \int_{t_{2j-1}}^{t_{2j}} [du]_{ij}(s+t\xi) \xi_i \xi_j \ \mathrm{d} t = I\epsilon(s,\xi) - I\chi \, ^s\epsilon(s,\xi),
\]
or,
\begin{equation}
\label{residual}
\sum_{j=1}^N \Bigg [ u_i(s + t \xi) \xi_i \Bigg ]_{t = t_{2j-1}}^{t = t_{2j}}= I\epsilon(s,\xi) - I\chi \, ^s\epsilon(s,\xi).
\end{equation}
The right hand side of this expression is the residual we obtain by comparing the LRT of $\epsilon$ to the LRT of ${^s}\epsilon$ masked to zero at all points outside $\Omega$ (easily computed from the result of \eqref{2DInvFormula} or \eqref{3DInvFormula}).  The left hand side represents corresponding differences in the longitudinal component of $u$ at the entry and exit points of the ray, and, numerically, it has been shown that this provides enough information to recover $u|_{\partial \Omega}$ up to infinitesimal rigid body motions \cite{wensrich2016bragg}. In this context, $u$ is interpreted as a displacement field and, in $\mathbb{R}^2$, an infinitesimal rigid body motion (isometry) is any vector field of the form $(c_1- x_2 \beta, c_2 + x_1 \beta)$ for scalar constants $\beta,c_1,c_2\ll1$.  Similarly, in $\mathbb{R}^3$ a rigid body motion is any vector field of the form $c+\beta \times x$ for suitably small $c,\beta \in \mathbb{R}^3$.

On this basis, we form the following theorem;

\begin{theorem}
\label{BdryThm}
	Let $u \in H^1(\mathcal{S}^1; \mathbb{R}^n),$ $ n=2,3$, and $\chi$ be the characteristic function for $\Omega$. The LRT of $\chi du$ determines $u|_{\partial\Omega}$ up to an infinitesimal rigid body motion on each component of $\partial \Omega$.
\end{theorem}

\begin{remark}This theorem will allow the recovery of $du$ within $\Omega$ through the following principle. Given $\epsilon={^s}\epsilon + du$, \eqref{Equilibrium} holds that
\begin{equation}
\label{ElasticModel}
    {\rm Div}(C : du)  = - {\rm Div}(C : {^s}\epsilon)
\end{equation}
within $\Omega$. Combining this with a boundary condition for $u$ obtained as above gives a complete boundary value problem that uniquely defines $u$ up to rigid body motions which are in the null space of $d$.

Note that this can only be achieved in the case where each component of $\Omega$ has a connected boundary; elastic strain following from relative rigid body motion of disconnected components of $\partial \Omega$ from the same component of $\Omega$ are in the null space of $I$.

\end{remark}

To prove Theorem \ref{BdryThm} we use \cite[Theorem 3.3]{CRMATH_2006}, which states the following.

\begin{theorem}\label{Amrouche}
    Let $B\subset\mathbb{R}^n,$ $ n=2,3$, be a simply connected domain, and let $f\in L^2(\mathcal{S}^2;B)$ be a rank-2 symmetric tensor field on $B$ that satisfies that the Saint Venant tensor $Wf$ vanishes in $B$ in the sense of distributions. Then there exists a vector field $u\in H^1(\mathcal{S}^1;B)$ such that $f=du$.
\end{theorem}

\begin{remark}
    The results in \cite{CRMATH_2006} are presented in $\mathbb{R}^3$ and from the perspective of the $\textnormal{\bf CURLCURL}$ operator. Nonetheless, there is a direct identification of $\textnormal{\bf CURLCURL}(f)$ and $Wf$ (\cite[Theorem 2.1 (b)]{CRMATH_2006}), and the proof of \cite[Theorem 3.3]{CRMATH_2006} can be quickly adapted to the 2 dimensional case. To prove the 2 dimensional version of \cite[Theorem 3.3]{CRMATH_2006}, only \cite[Theorem 3.2-step (iii)]{CRMATH_2006} needs to be slightly modified, using the following differential identities in $\mathbb{R}^2$ instead of \cite[Theorem 2.1 (a)]{CRMATH_2006}.
\end{remark}

\begin{lemma}
    Let $\Phi(f)=[W f]_{1122}$ be the only independent component in $\mathbb{R}^2$ of the Saint Venant tensor $Wf$, for $f$ a rank-2 symmetric tensor field. The following identities hold in the sense of distributions for $u$ a vector field and $f$ a rank-2 symmetric tensor field in $\mathbb{R}^2$:
    \begin{enumerate}
        \item $\Phi(du)=0$.
        \item $\Delta(\textnormal{tr}(f)) = \Phi(f)+\textnormal{Div}(\textnormal{Div}(f))$.
        \item $\Delta f_{ij}+\frac{\partial^2}{\partial x_i \partial x_j} \textnormal{tr}(f) = \delta_{ij}\Phi(f) +\frac{\partial}{\partial x_i}[\textnormal{Div} (f)]_j + \frac{\partial}{\partial x_j}[\textnormal{Div} (f)]_i$, where $i,j=1,2$ and $\delta_{ij}$ is the Kronecker tensor.
    \end{enumerate}
\end{lemma}

The proof of Theorem \ref{BdryThm} is as follow.

\begin{proof}[Proof of Theorem \ref{BdryThm}]
Let $n=2,3$. First assume that $\Omega\subset\mathbb{R}^n$ is a simply connected bounded domain with Lipschitz boundary (such that $\partial\Omega$ is the boundary of $\Omega$ and that of $\mathbb{R}^n\setminus \overline{\Omega}$).

Let $u,v \in H^1(\mathcal{S}^1;\mathbb{R}^n)$ be such that the LRT of $\chi du$ equals the LRT of $\chi dv$. Then $f =\chi d(u-v) \in L^2(\mathcal{S}^2;\mathbb{R}^n)$ and the LRT of $f$ vanishes everywhere, which by Sharafutdinov \cite[Theorem 2.12.3]{sharafutdinov2012integral} implies that $W f =0$ in all of $\mathbb{R}^n$.

Let $B$ be a ball of radius $R$ centred at the origin be large enough such that $\overline\Omega \subset B$. Since $\Omega$ is simply connected then $\Omega^+:=B\setminus\overline\Omega$ is connected and its boundary contains $\partial\Omega$.

By restricting $f$ to $B$ and using Theorem \ref{Amrouche}, there exist $w\in H^1(\mathcal{S}^1;B)$ such that $dw = f$ in $B$.

Inside of $\Omega$, $dw=f$ means $dw=\chi d(u-v)= d(u-v)$, in other words $d(w-(u-v))=0$ in $\Omega$. Since $\Omega$ is a connected domain, this implies that $w-(u-v)=g$ in $\Omega$, where $g$ is a rigid body motion \cite{ciarlet2018vector}.

On the other hand, outside of $\Omega^+$, $dw=f$ means $dw=\chi d(u-v)= 0$ in $\Omega^+$. Since $\Omega^+$ is a connected domain, this implies that $w=g_+$ in
 $\Omega^+$, where $g_+$ is a rigid body motion.
 
Since $w,u,v\in H^1(\mathcal{S}^1;B)$ and $\partial\Omega$ is boundary of $\Omega \subset B$ and part of the boundary of $\Omega^+\subset B$, we get
\begin{equation*}
w-(u-v)=g \textnormal{ and } w=g_+ \textnormal{ at } \partial\Omega,
\end{equation*} 
in the $H^{1/2}(\mathcal{S}^1;\partial\Omega)$ sense. From here, $u-v=g_+-g$ at $\partial \Omega$, i.e. $u|_{\partial\Omega}$ and $v|_{\partial\Omega}$ are equal up to a rigid body motion.

When $\Omega$ is not simply connected, the only difference in the argument is that $\Omega^+$ is not connected and it has one corresponding component for each component of $\partial\Omega$. Therefore, instead of a single $g_+$, there will be one independent rigid body motion $g_i$ for each component of $\partial\Omega$.
\end{proof}

\section{Reconstruction of boundary displacement}

\subsection{Numerical approach}
\label{NumericalScheme}
 
A 2D numerical scheme was constructed to compute the boundary displacement from the LRT of a potential from the perspective of \eqref{residual}.  This was approached largely inline with prior work by Wensrich \textit{et al.} \cite{wensrich2016bragg} and Hendriks \textit{et al.} \cite{hendriks2017bragg}.  

The process is based on a discretised version of a given boundary as a set of nodal points $\{x^i\}_{i=1}^{l} \in \partial\Omega$ with associated unknown displacement vectors $\{u^i\}_{i=1}^{l}$.  From this, the boundary displacement at an arbitrary point can be expressed as a linear interpolation between neighbouring nodes.  With reference to Figure \ref{Fig:NumericalScheme} and the left hand side of \eqref{residual}, this allows the following approximation on the given ray
\begin{equation}
    \label{rayapprox}
    (u^q-u^p)\approx\alpha^qu^j+(1-\alpha^q)u^{j+1}-\alpha^pu^i-(1-\alpha^p)u^{i+1},
\end{equation}
where
\[
\alpha^p=\frac{||x^p-x^{i+1}||}{||x^i-x^{i+1}||} \text{ and } \alpha^q=\frac{||x^q-x^{j+1}||}{||x^j-x^{j+1}||},
\]
with a natural extension in the case of multiple entry and exit points.

\begin{figure}
\begin{center}
    \includegraphics[width=0.7\linewidth]{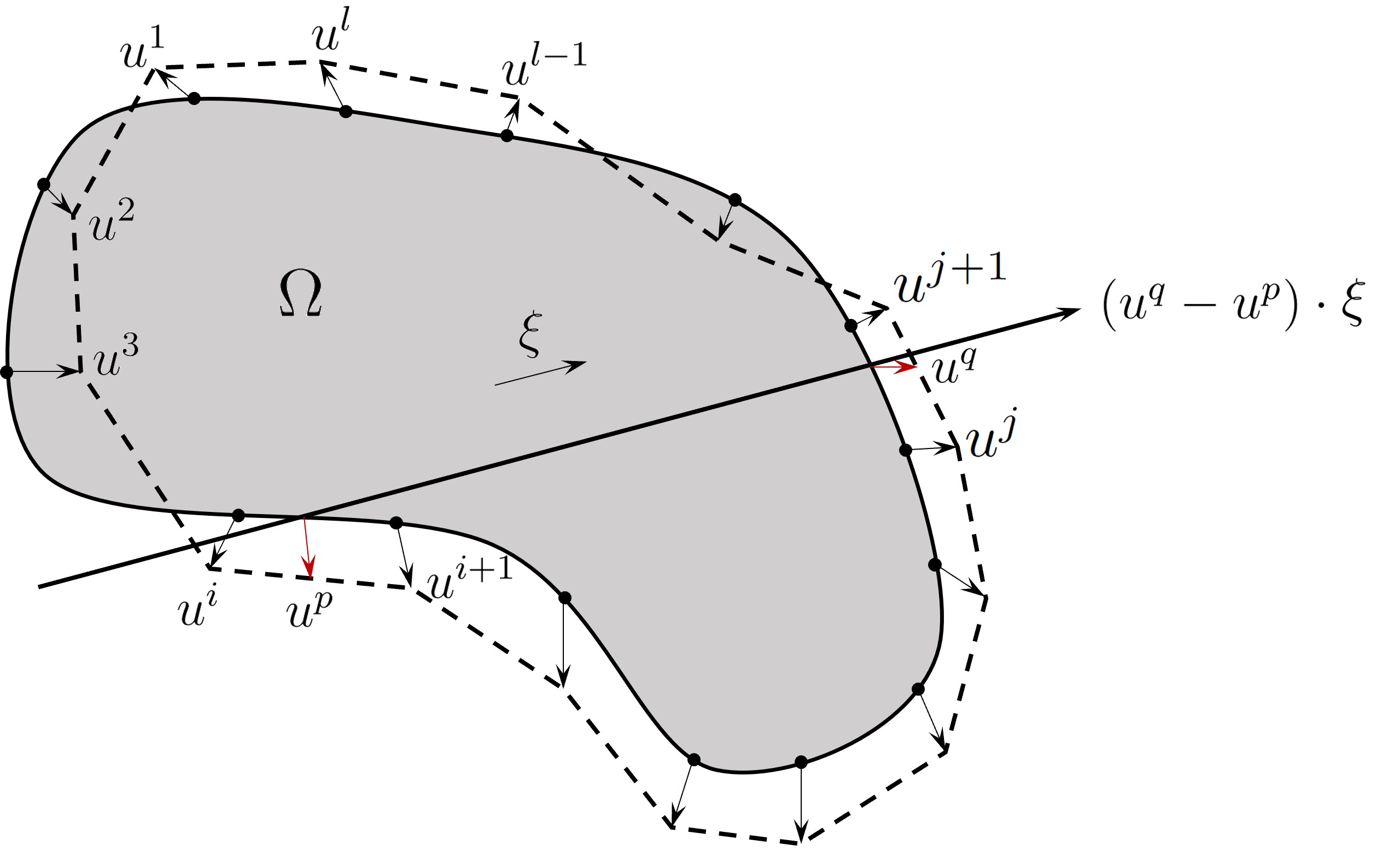}
    \caption{\label{Fig:NumericalScheme}A schematic of the numerical scheme to reconstruct $u$ on a discretised boundary.}
\end{center}
\end{figure} 

From this we can construct a system of linear equations that can be solved for the unknown nodal displacements; one equation for each ray that intersects the boundary.  For $M$ rays, this has the form
\begin{equation}
    \label{AU=R}
    A_{[M\times 2l]}U_{[2l \times 1]}=R_{[M \times 1]},
\end{equation}
where $A$ is a matrix of coefficients with each row constructed in line with \eqref{rayapprox}, $U$ is a vector containing unknown components of the nodal displacements and $R$ is a vector containing values of the LRT for each corresponding ray.  With sufficient rays ($M>2l$) the system can be over determined, however the insensitivity to rigid body motion will cause $A$ to be rank deficient by exactly three times the number of separate boundary components regardless of the number of rays.

Once constructed, \eqref{AU=R} can be solved using the Moore-Penrose pseudo-inverse of $A$ to yield $U$ along with an arbitrary rigid body motion chosen to minimise $||U||$ subject to \eqref{AU=R}.

\subsection{Implementation and numerical stability}

The scheme was implemented in MATLAB utilising the `\texttt{polyxpoly}' intrinsic function to locate points of intersection between incident rays and boundary segments.  

Numerical stability of the process was investigated by examining the singular values of $A$ for a range of different boundaries of one or more separate components.  Each of these boundary components was represented by an ordered set of 1000 points in the plane.  In each case, the matrix $A$ was constructed based on 500 projections uniformly spaced over one rotation with each projection consisting of multiple rays with an equal spacing of 0.03 units.

Figures \ref{LobesSV} to \ref{ThreeDisks} show the results of this process in terms of the range of singular values ordered from highest to lowest.  In each case the singular values were approximately in the range of 0.1 to 10 except for a number of `zero' singular values of the order of $10^{-15}$.  As expected, the number of zero-singular values corresponded to three times the number of separate boundary components.  Outside of these zero-singular values, the condition number of $A$ was very low; of the order of $10^2$.

\begin{figure}
\begin{center}
    \includegraphics[width=0.69\linewidth]{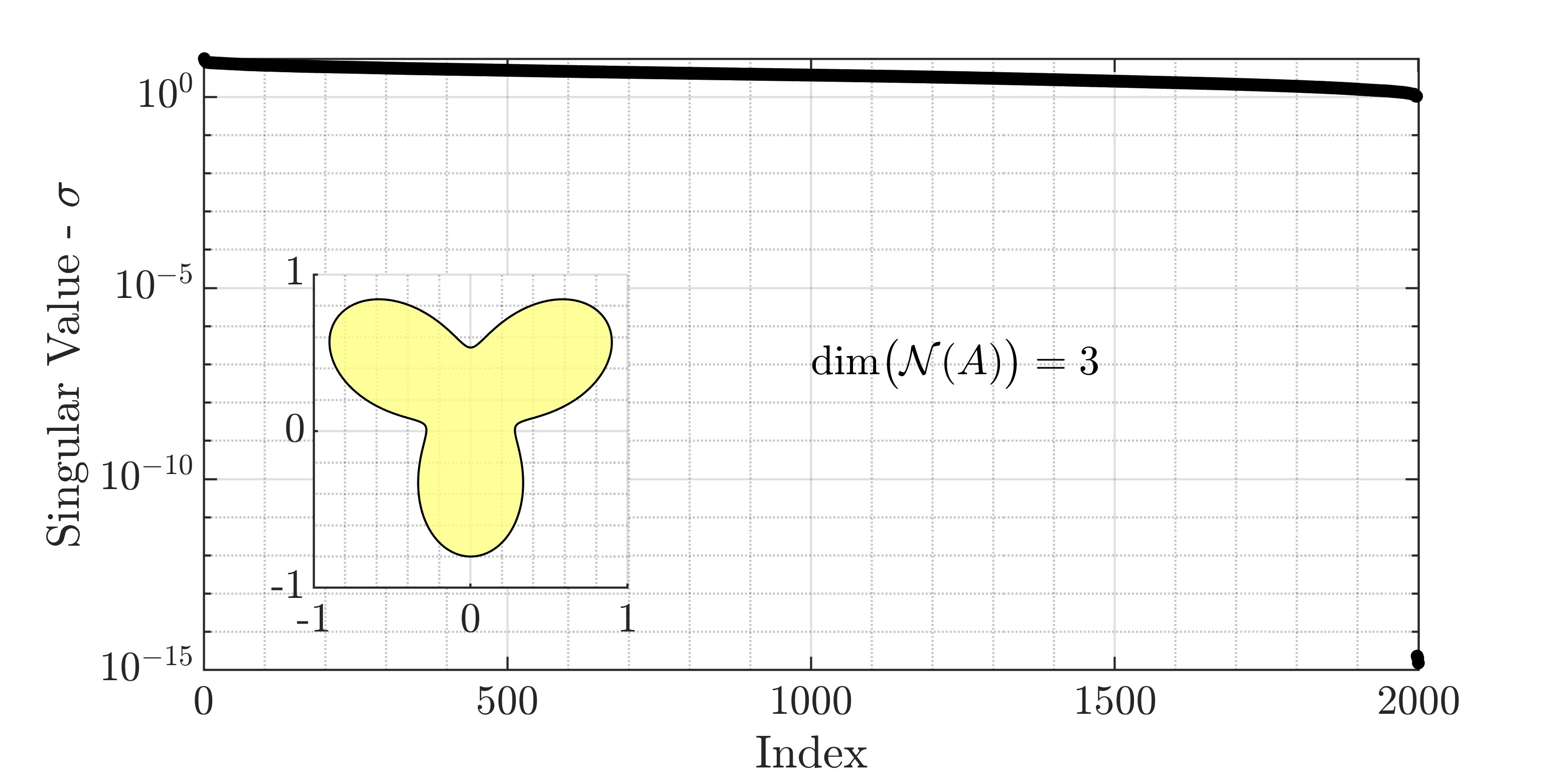} \\
    \includegraphics[width=0.3\linewidth]{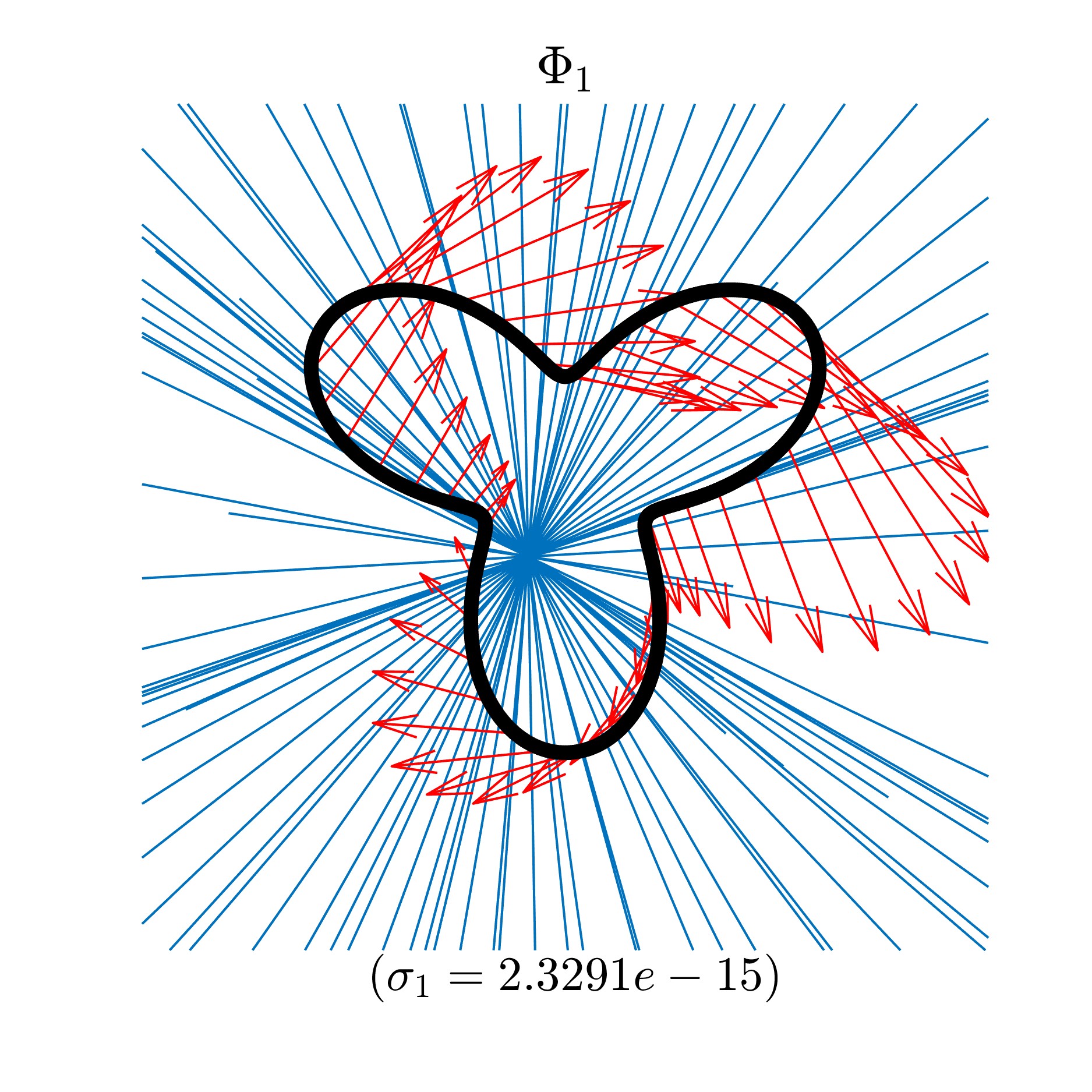}
    \includegraphics[width=0.3\linewidth]{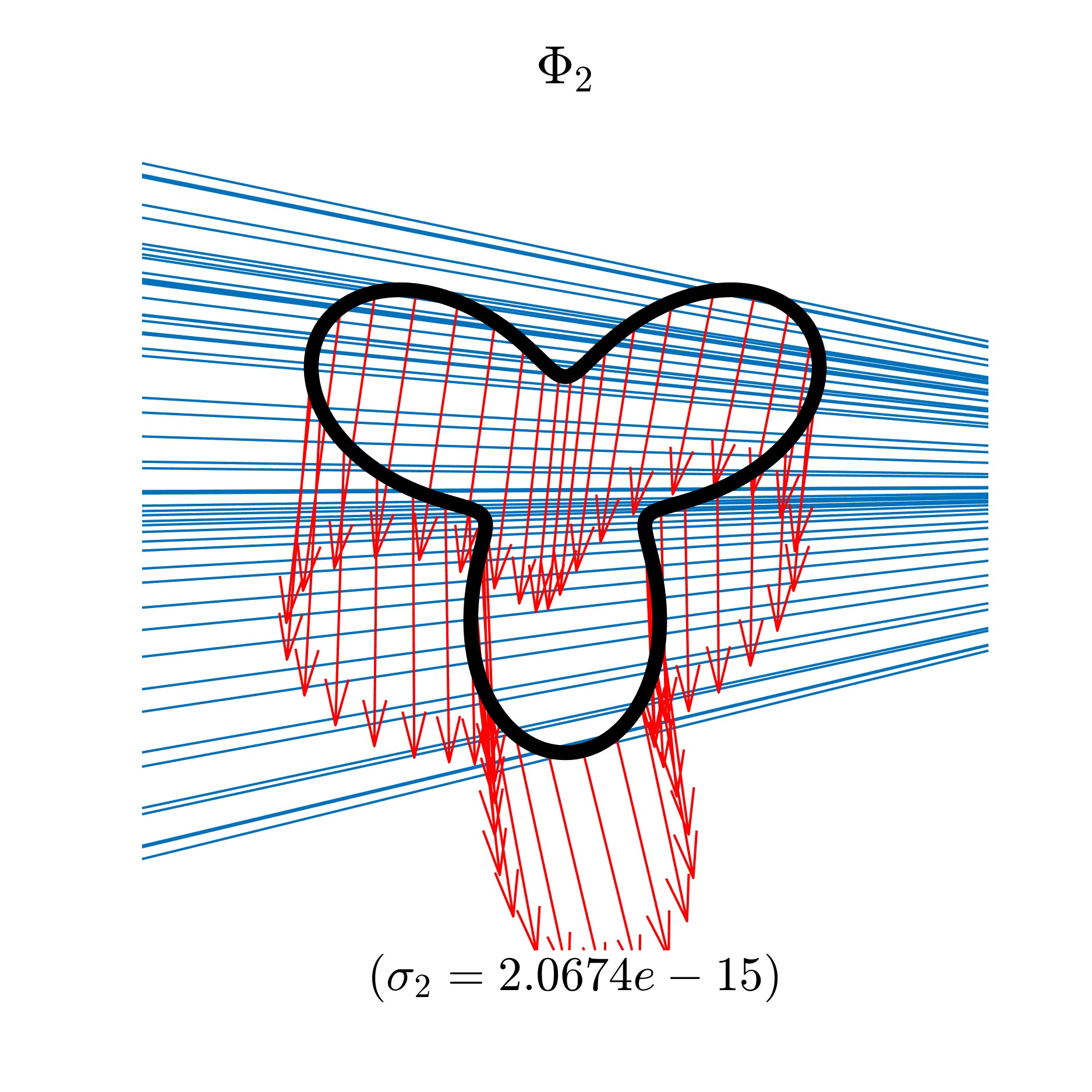}
    \includegraphics[width=0.3\linewidth]{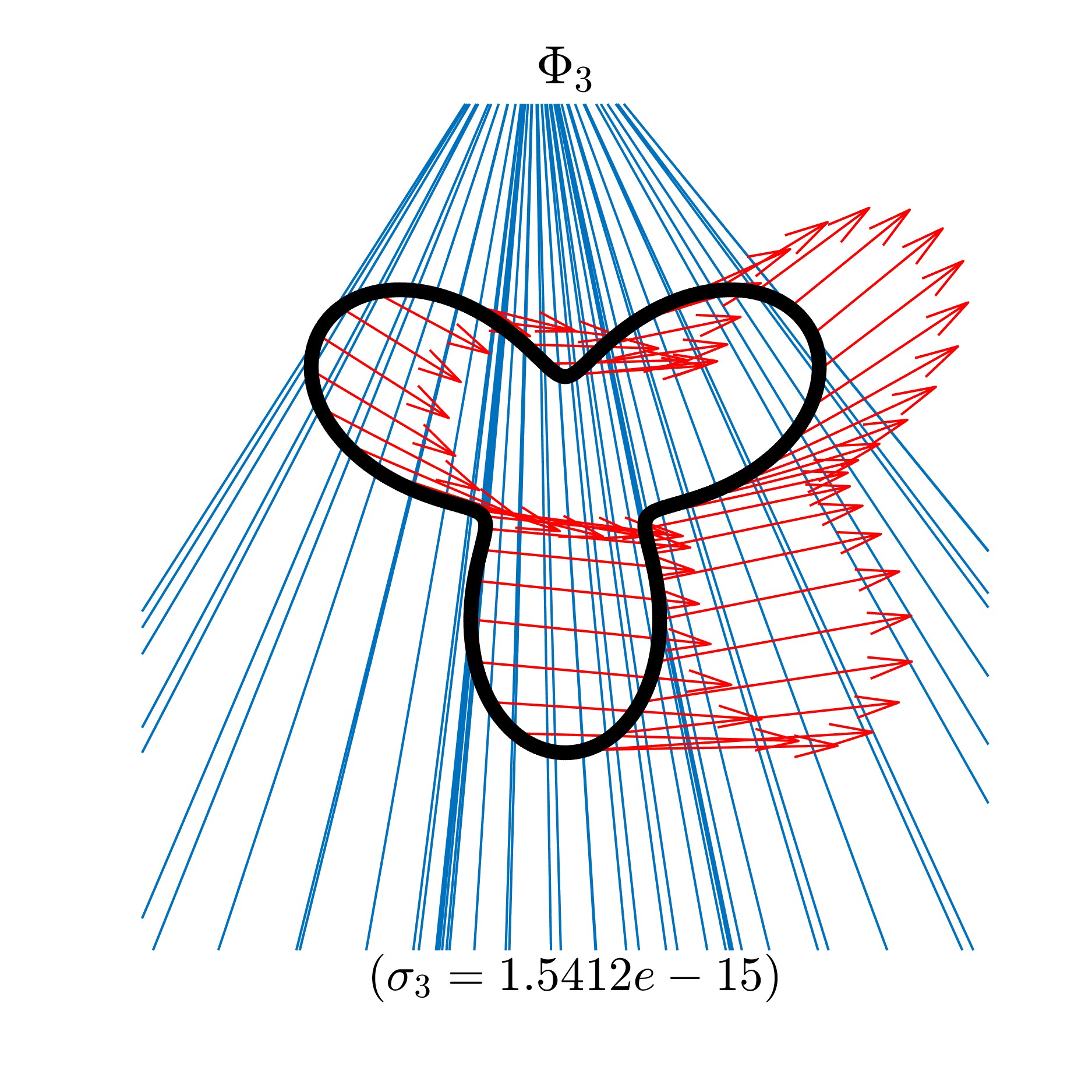}
    \caption{Singular value decomposition of $A$ for a non-convex 3-lobed shape.  Top: Singular values of $A$ listed in decreasing order.  Bottom: Zero-singular vectors of $A$.  }
    \label{LobesSV}
\end{center}
\end{figure} 

Also shown in Figures \ref{LobesSV} to \ref{ThreeDisks} are the singular vectors corresponding to the zero singular values of $A$; representing a possible basis for the null-space of $\mathcal{N}(A)$.  These are depicted with red arrows showing the displacement of corresponding black boundary nodes and blue lines drawn normal to the red displacement vectors.  In all cases, convergence of the blue lines to a single point (sometimes at infinity) for each separate boundary component was observed. 
These points represent instant centres of rotation for corresponding rigid body motions of the boundaries.

\begin{figure}
\begin{center}
    \includegraphics[width=0.69\linewidth]{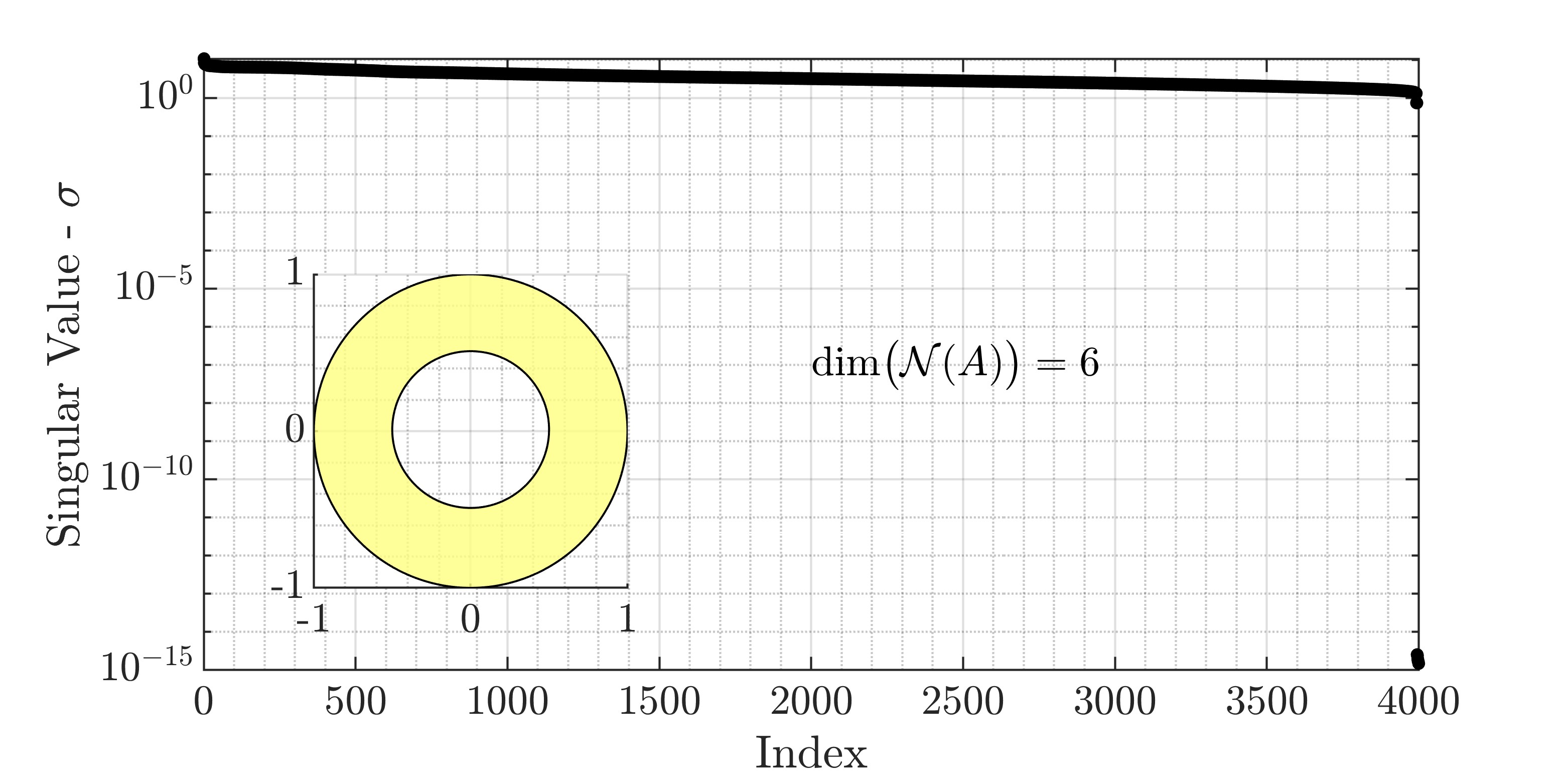} \\
    \includegraphics[width=0.3\linewidth]{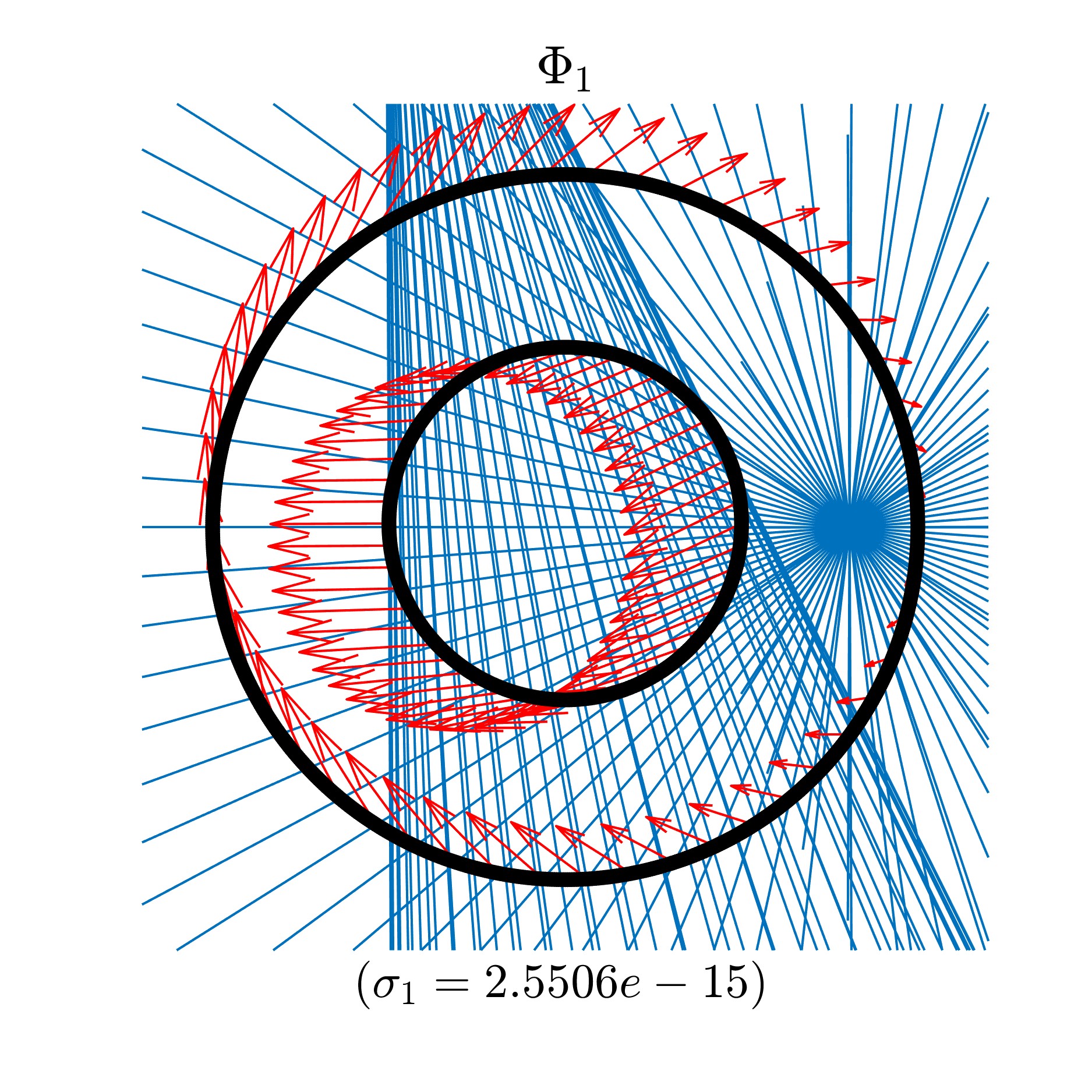}
    \includegraphics[width=0.3\linewidth]{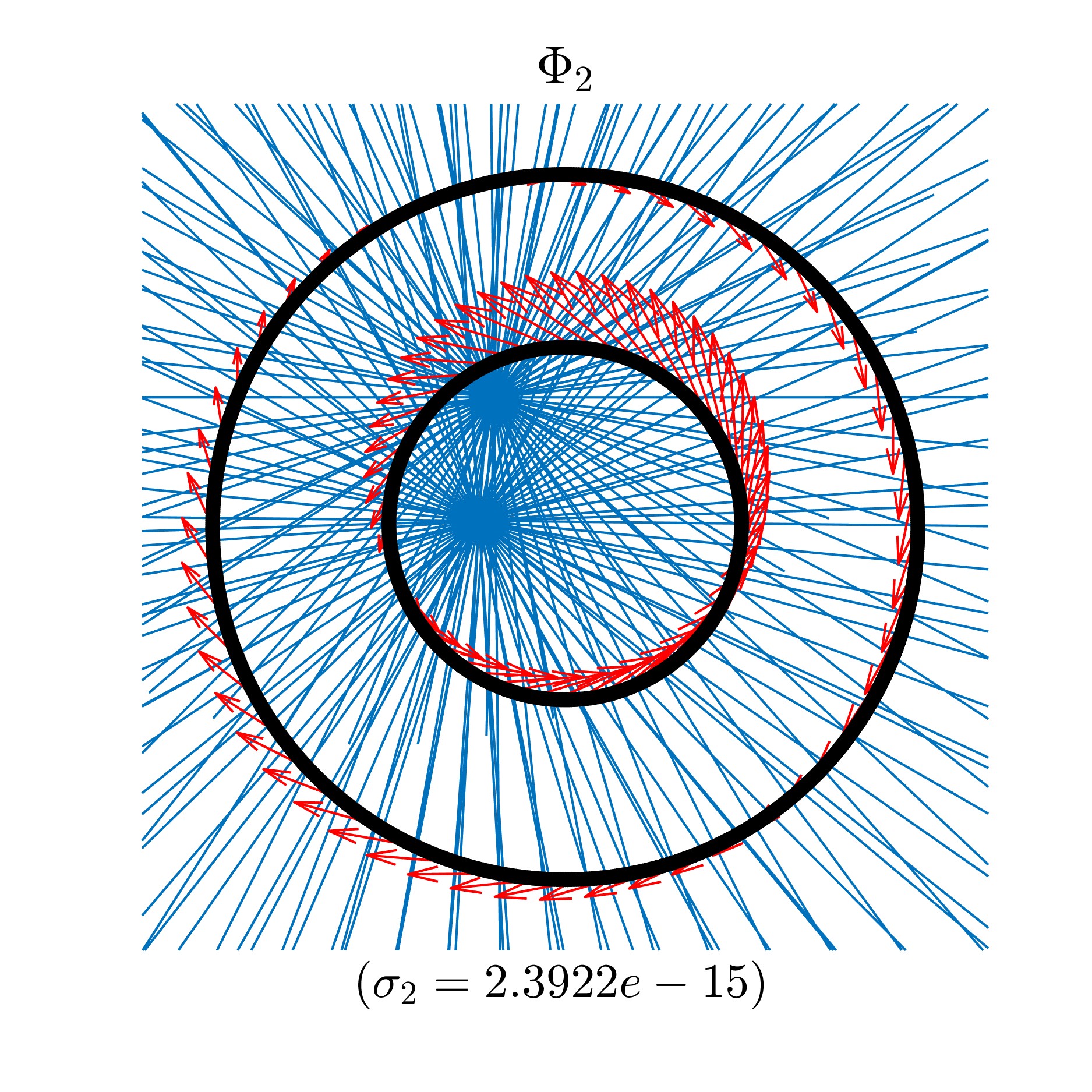}
    \includegraphics[width=0.3\linewidth]{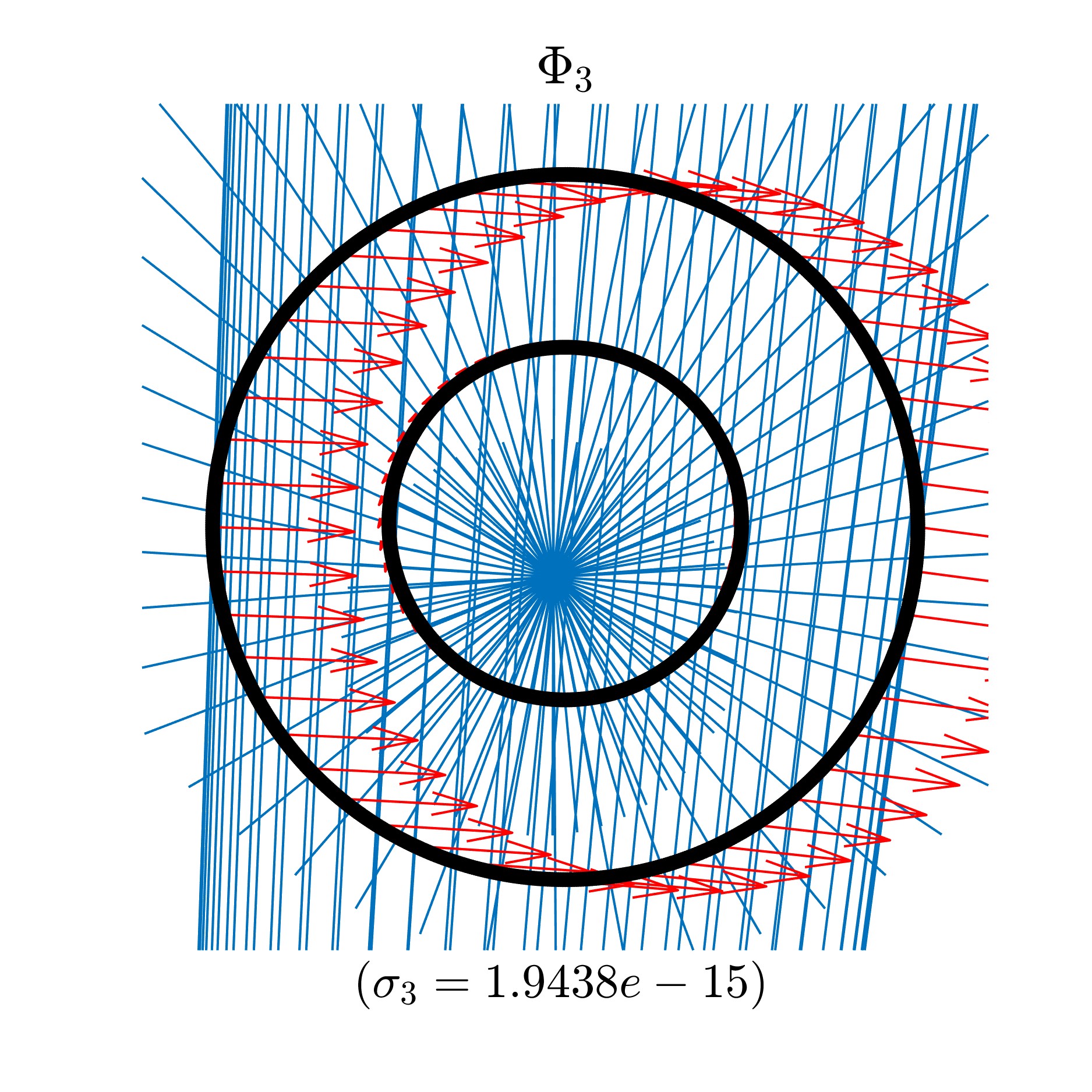}
    \includegraphics[width=0.3\linewidth]{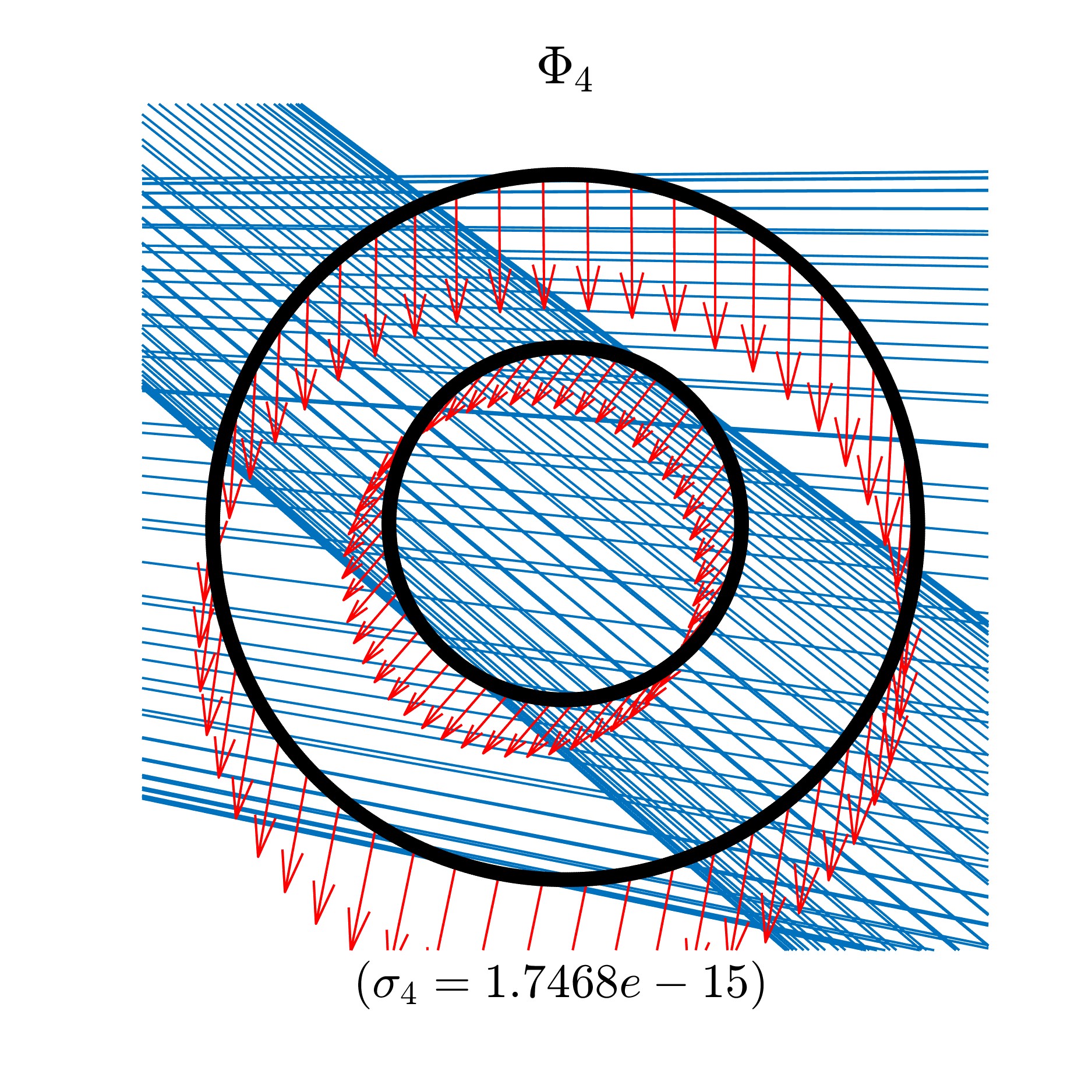}
    \includegraphics[width=0.3\linewidth]{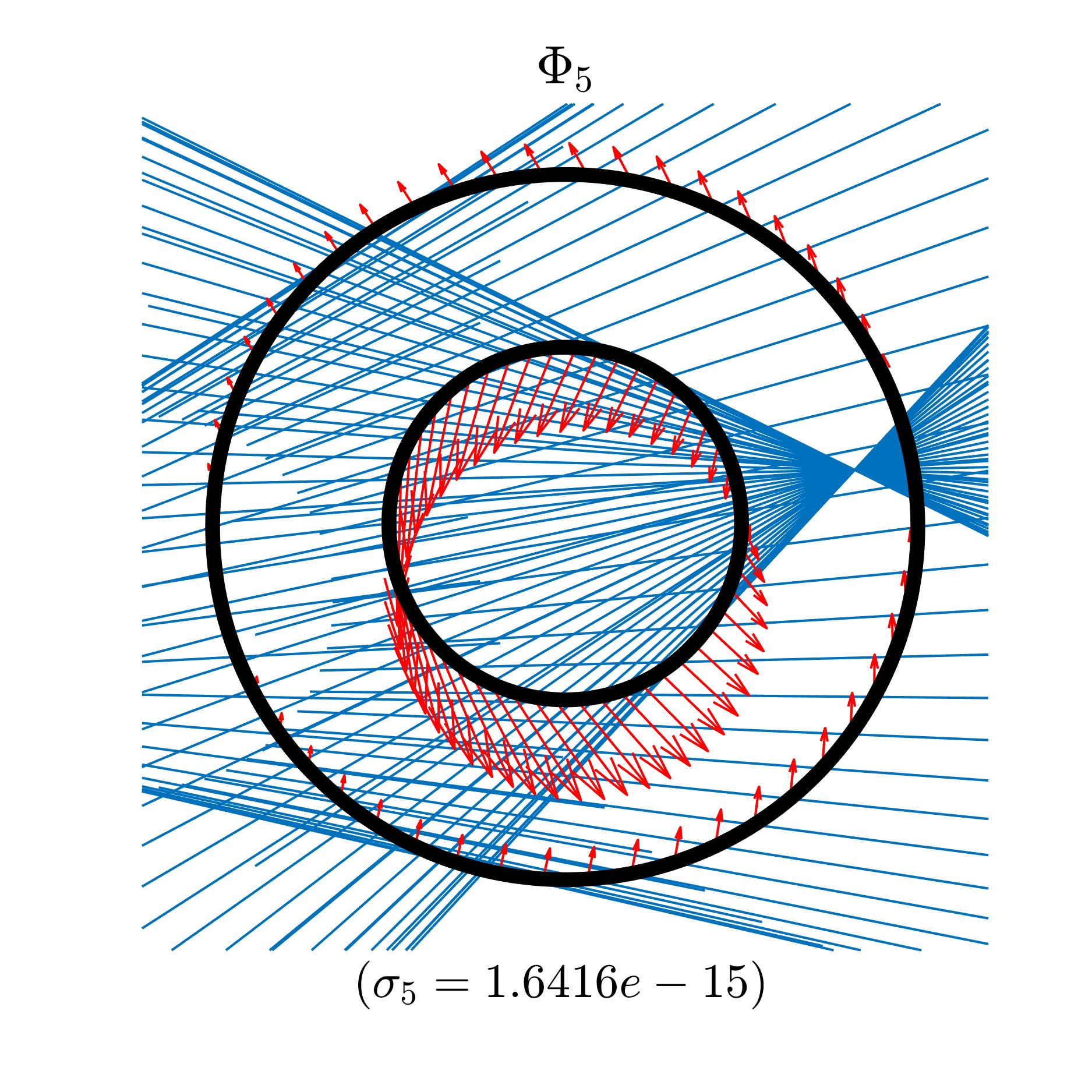}
    \includegraphics[width=0.3\linewidth]{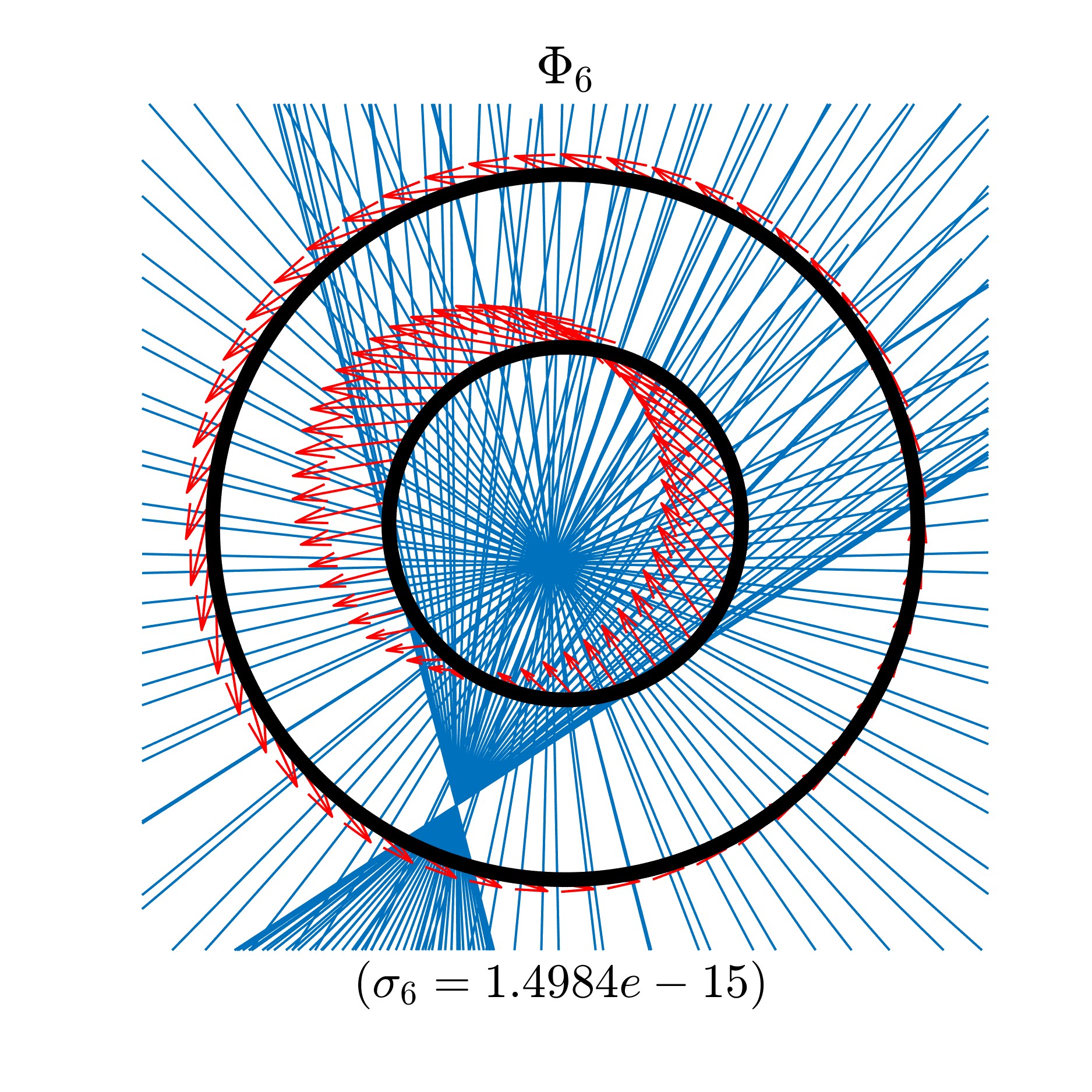}
    \caption{Singular value decomposition of $A$ for a system consisting of a disk with a circular hole.  Top: Singular values of $A$ listed in decreasing order.  Bottom: Zero-singular vectors of $A$.  }
    \label{DiskHoleSV}
\end{center}
\end{figure}

\begin{figure}
\begin{center}
    \includegraphics[width=0.69\linewidth]{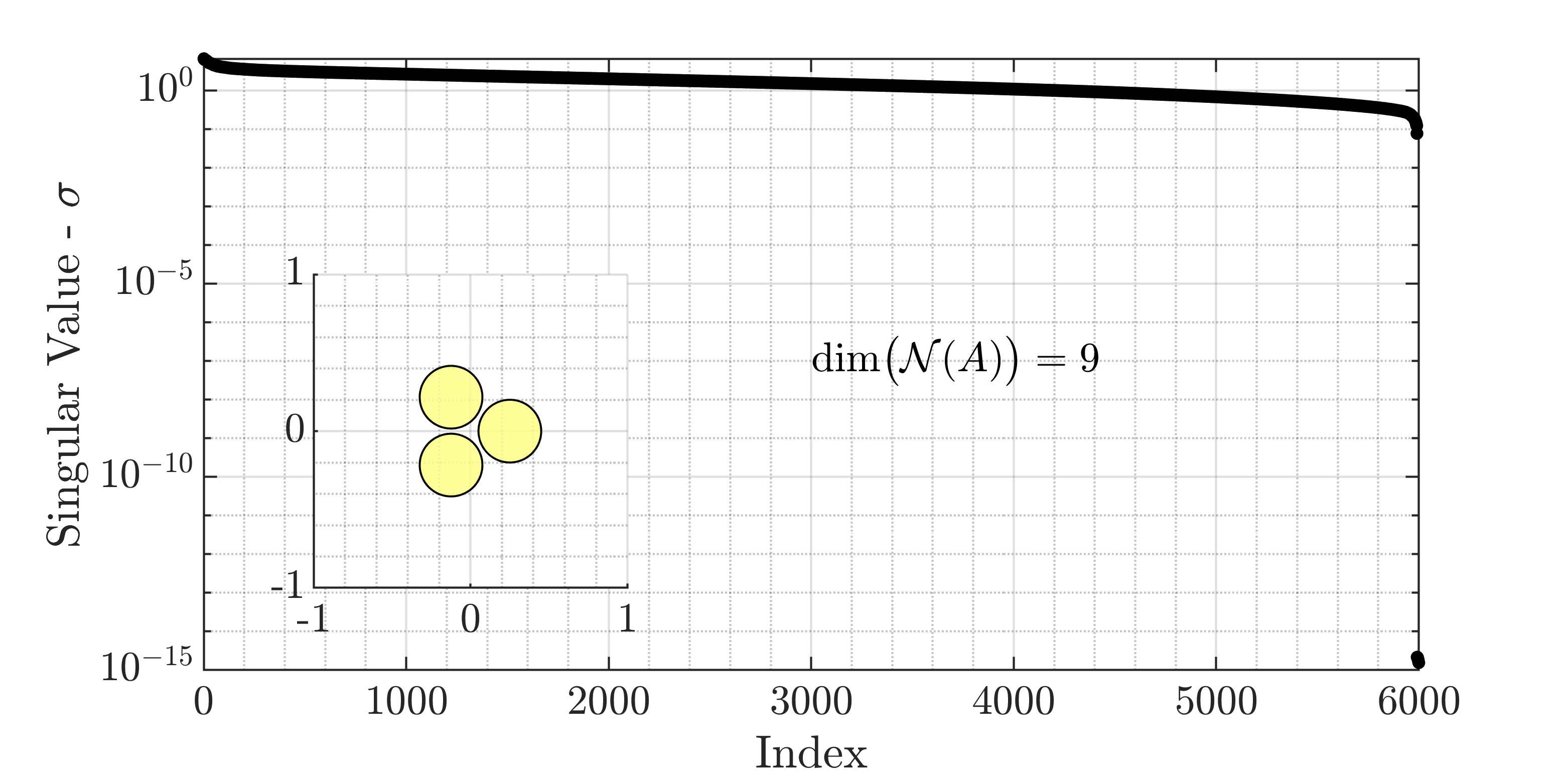} \\
    \includegraphics[width=0.3\linewidth]{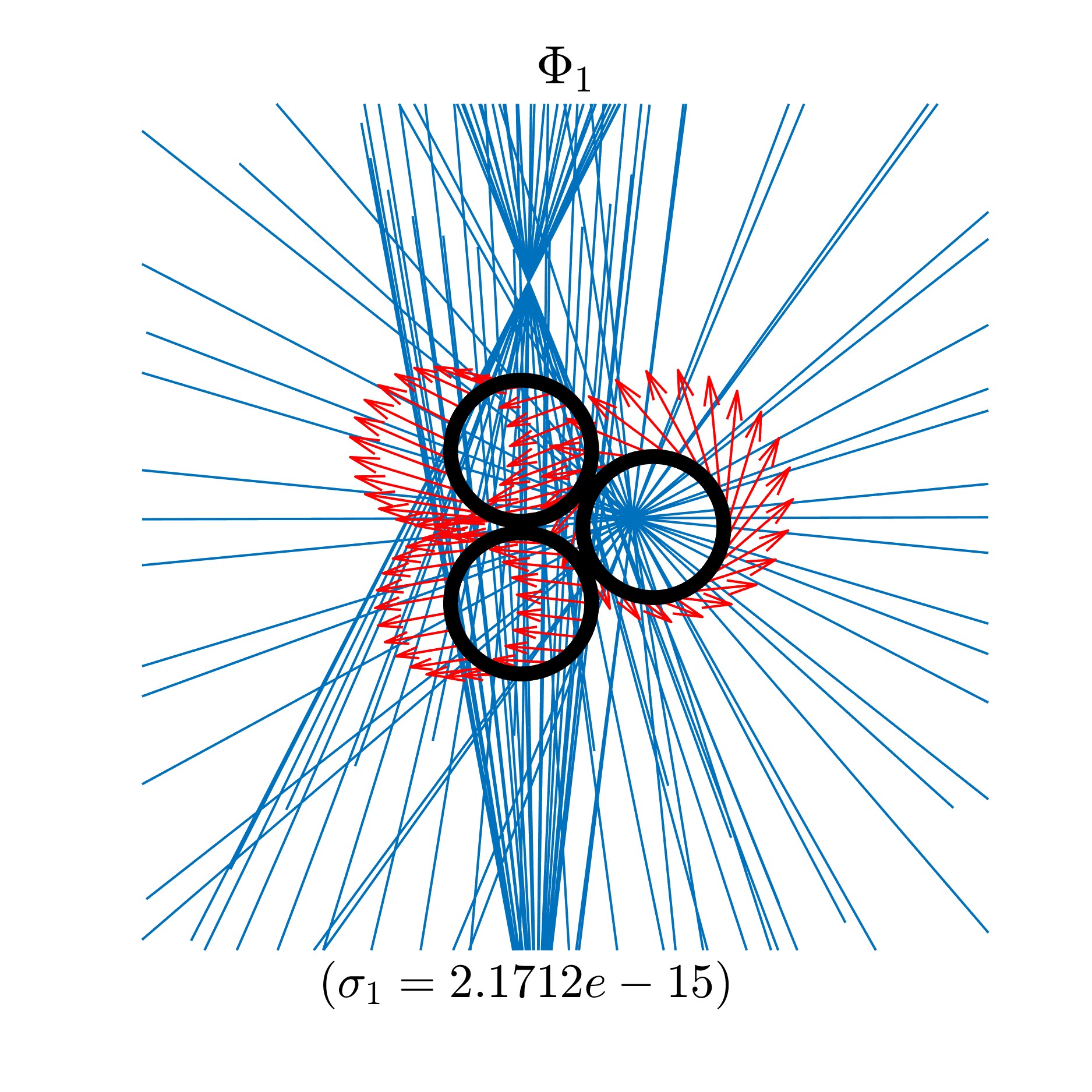}
    \includegraphics[width=0.3\linewidth]{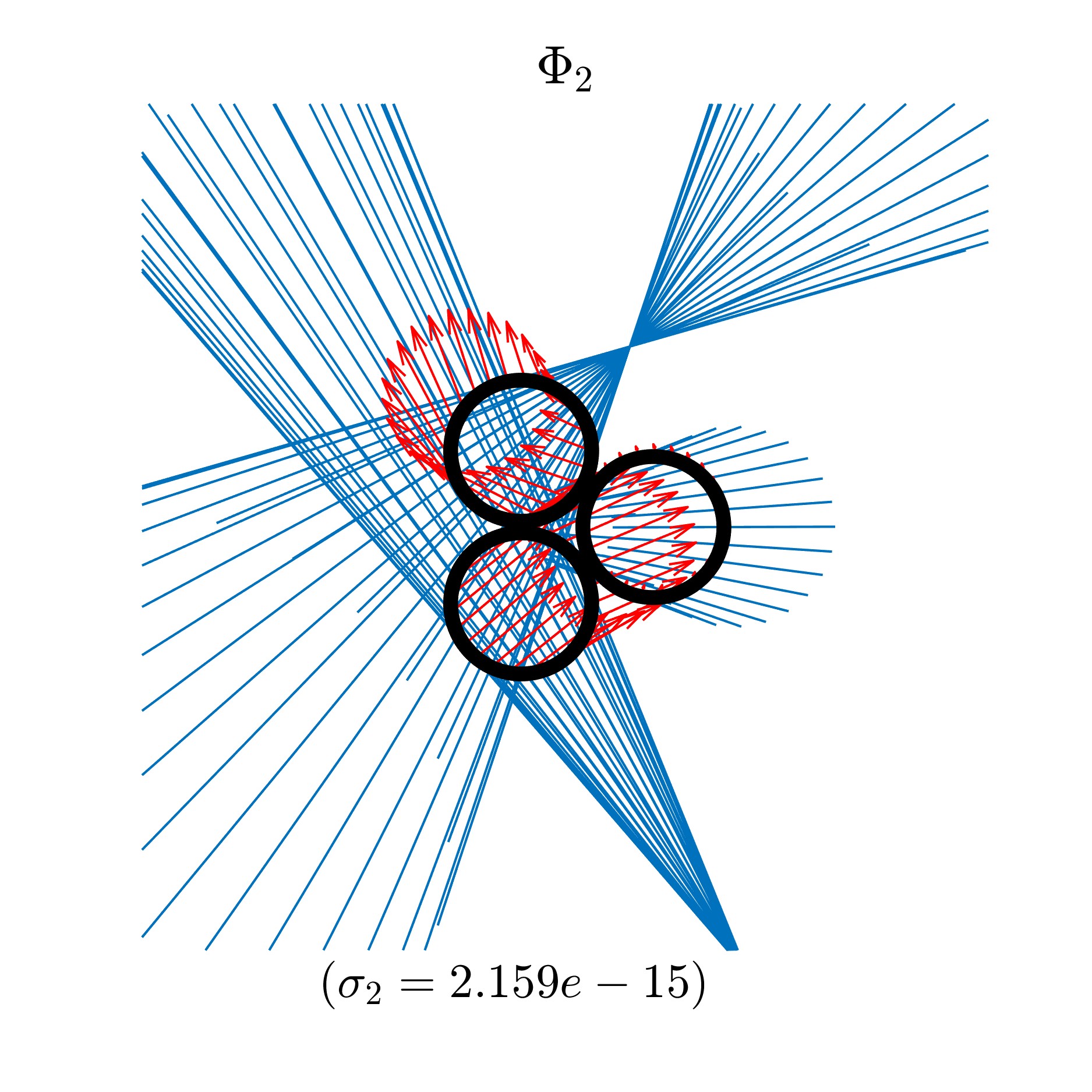}
    \includegraphics[width=0.3\linewidth]{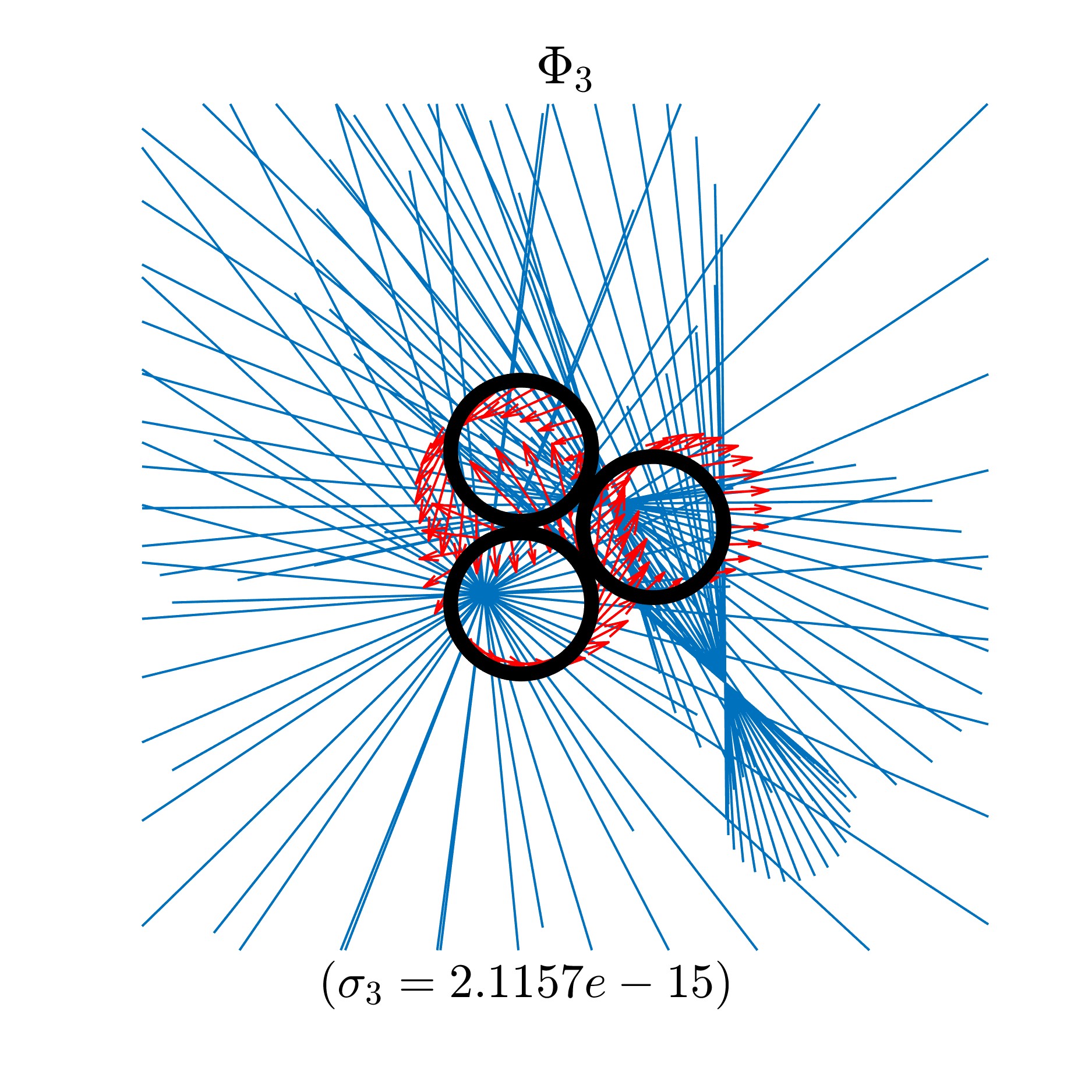}
    \includegraphics[width=0.3\linewidth]{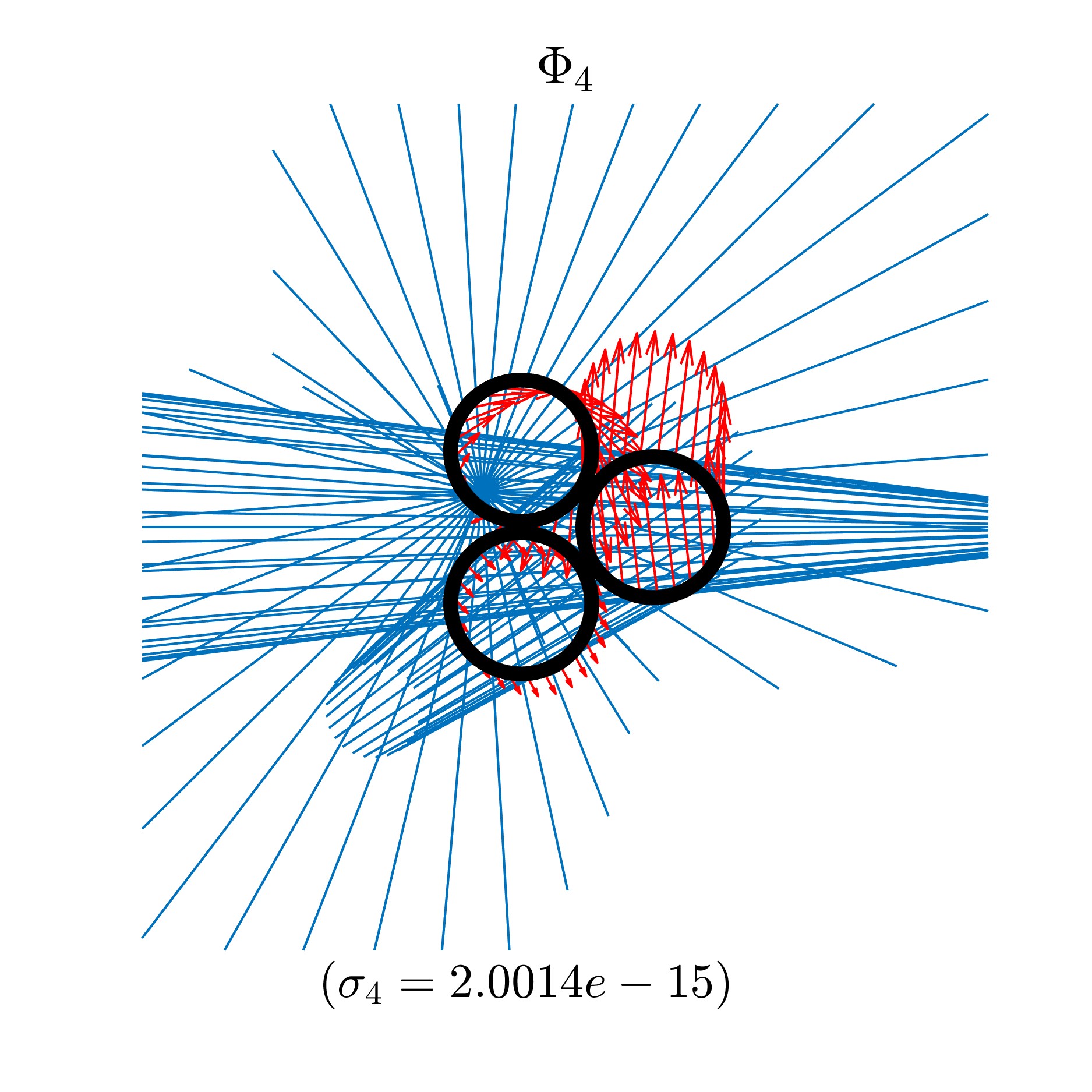}
    \includegraphics[width=0.3\linewidth]{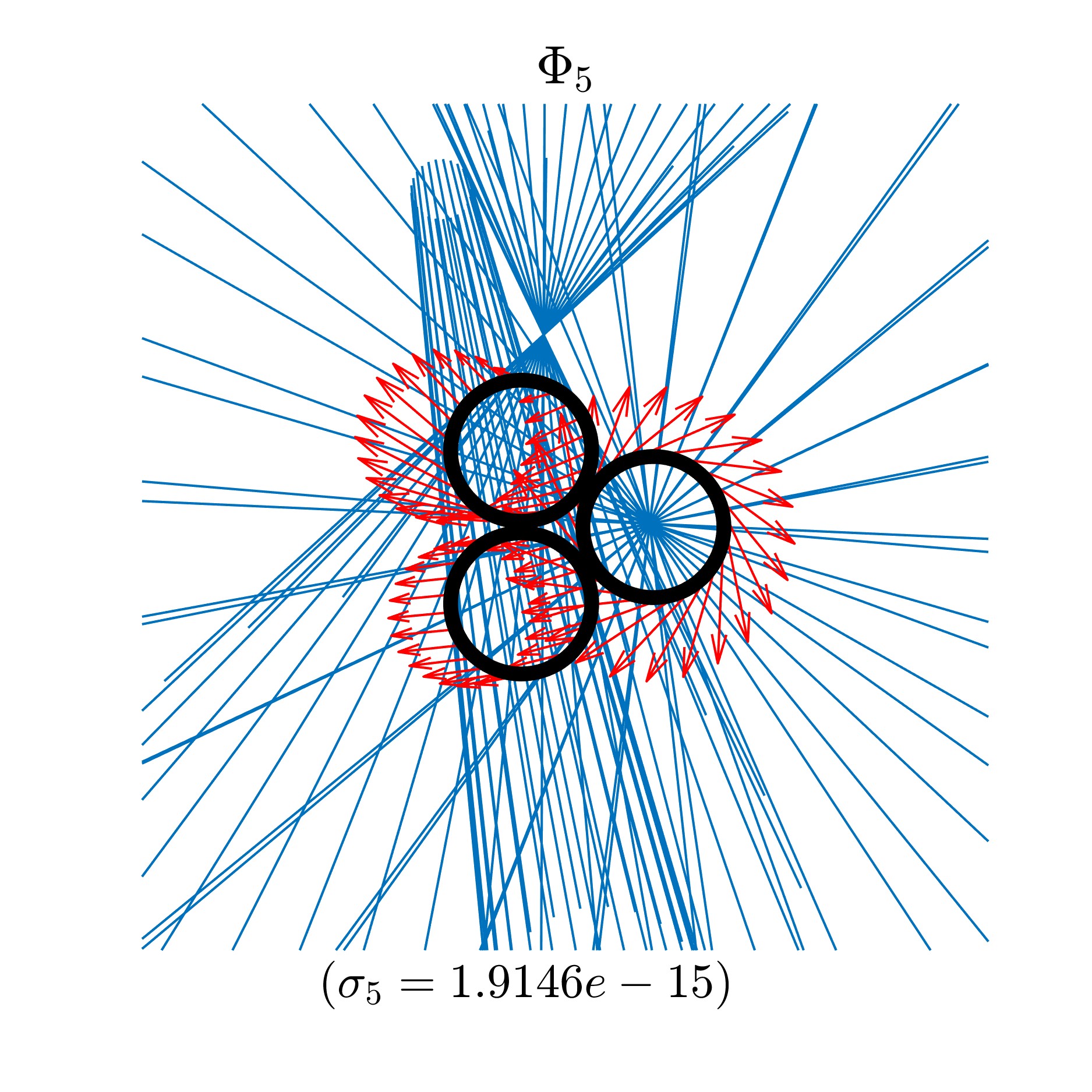}
    \includegraphics[width=0.3\linewidth]{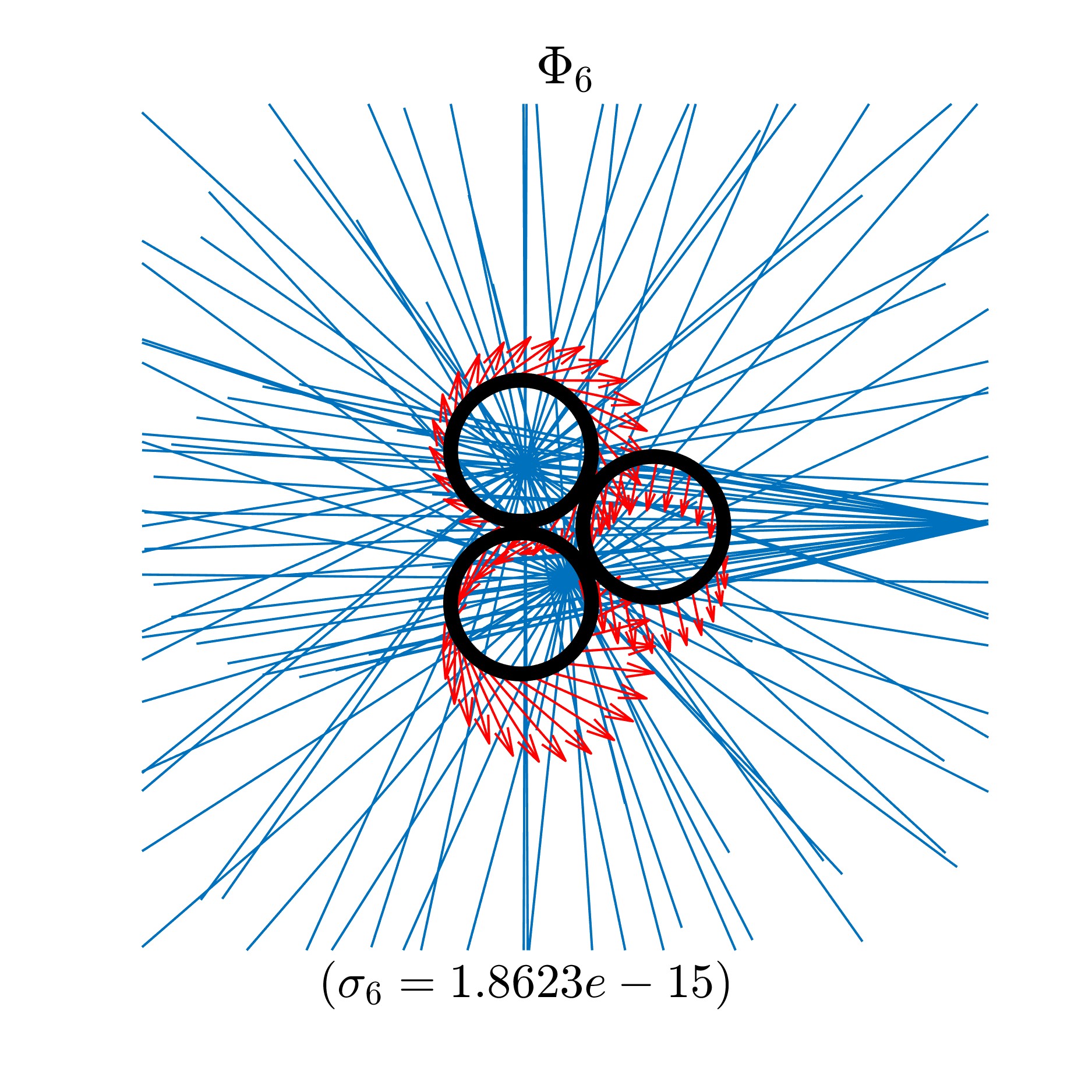}
    \includegraphics[width=0.3\linewidth]{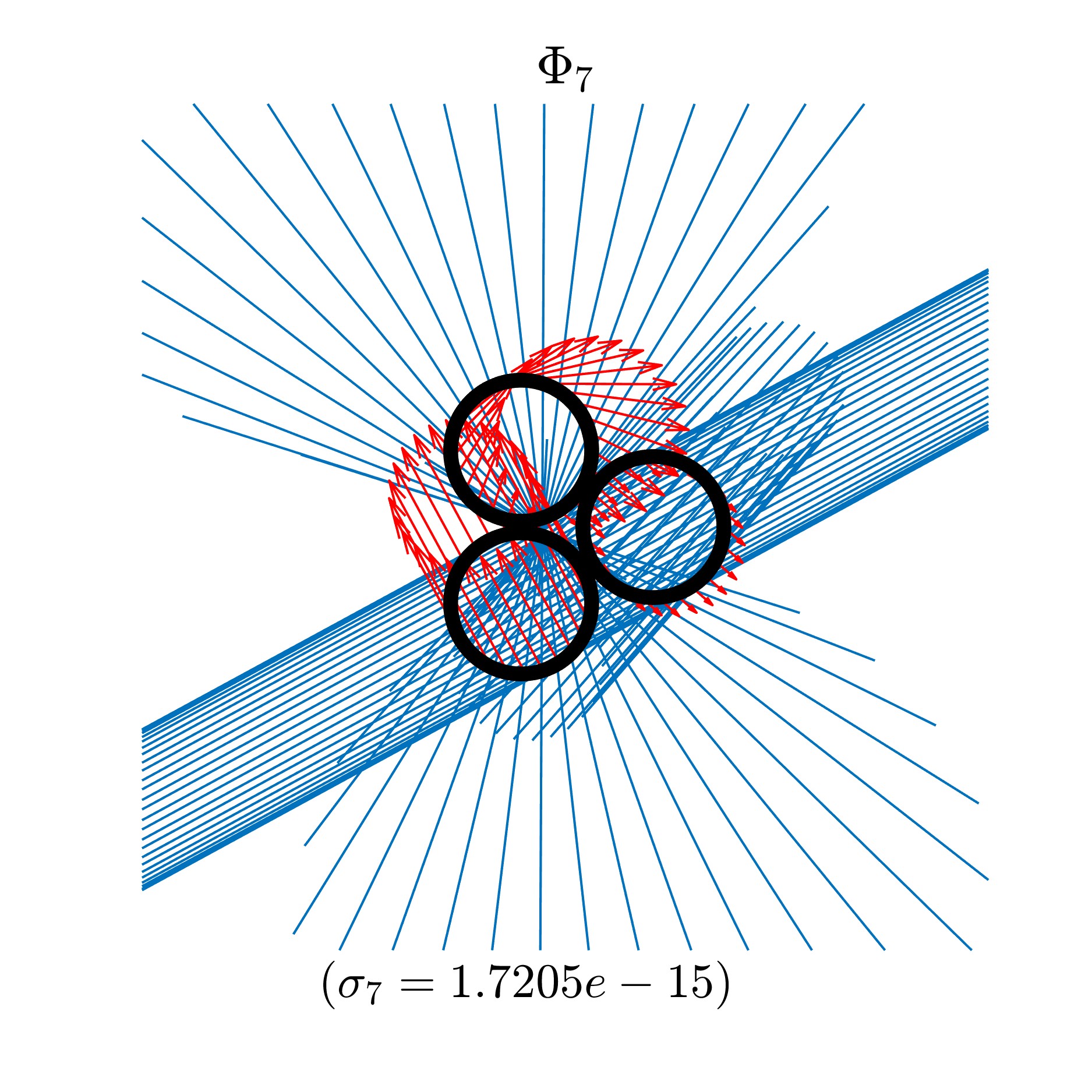}
    \includegraphics[width=0.3\linewidth]{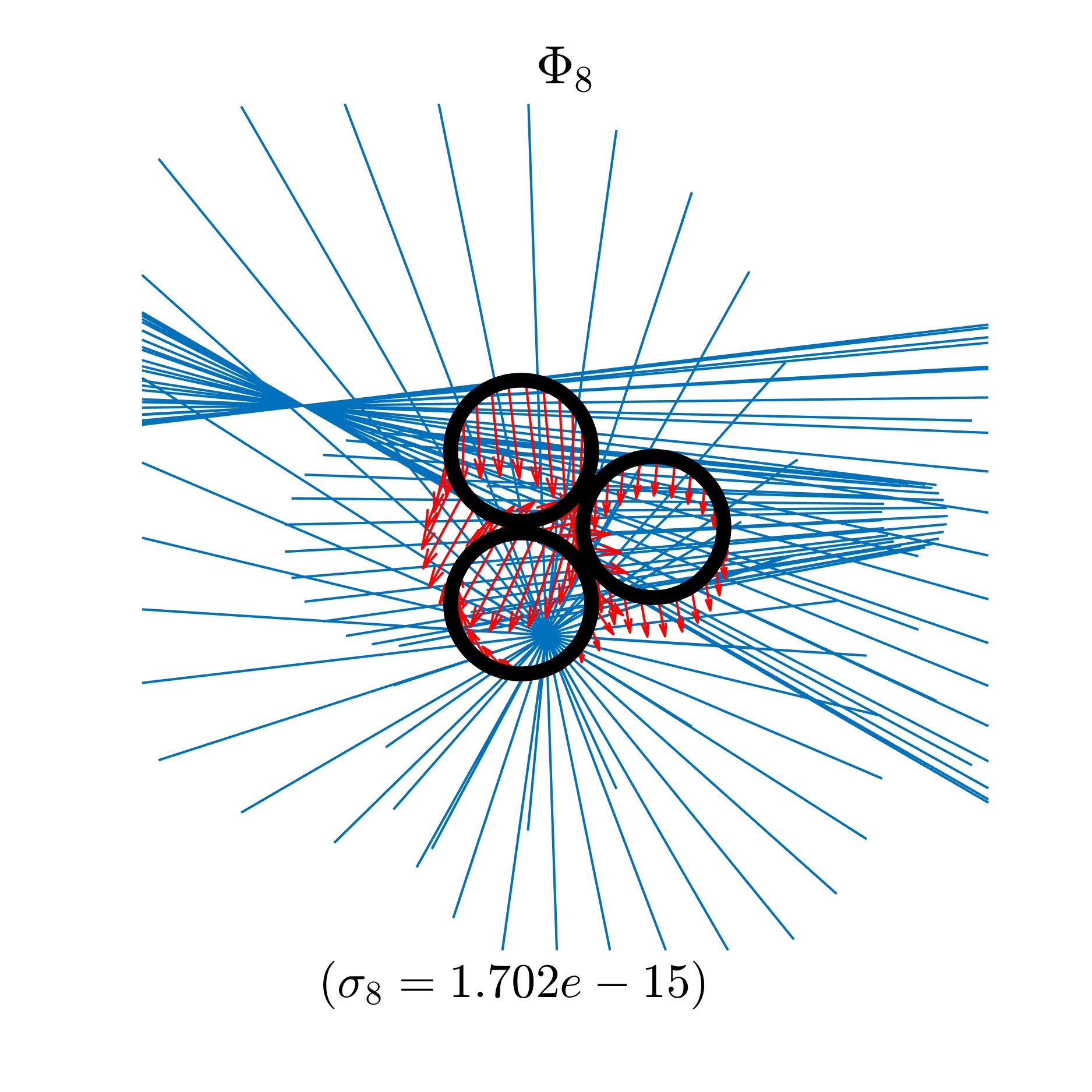}
    \includegraphics[width=0.3\linewidth]{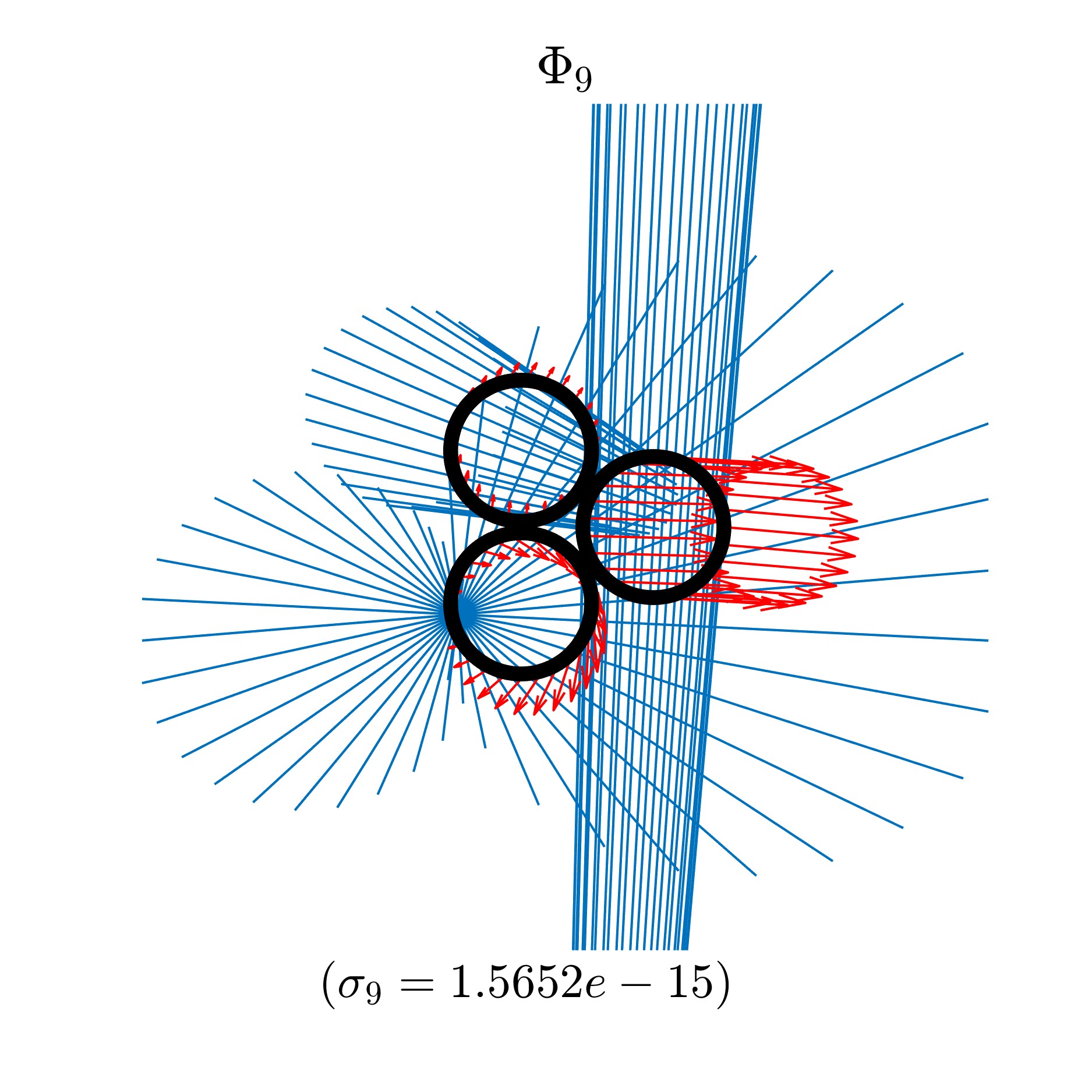}
    \caption{Singular value decomposition of $A$ for a system consisting of three unconnected disks.  Top: Singular values of $A$ listed in decreasing order.  Bottom: Zero-singular vectors of $A$. }
    \label{ThreeDisks}
\end{center}
\end{figure}

\section{Reconstruction of $\epsilon$ from $I\epsilon$}

\subsection{Numerical algorithm}
\label{algorithm}

Based on the above, numerical reconstruction of a plane-stress\footnote{A plane-stress state is an embedding of a three-dimensional stress state into $\mathbb{R}^2$ where $\sigma_{i3}=0 \enspace \forall i\in\{1,2,3\}$.  Plane-strain is a similar condition where $\epsilon_{i3}=0 \enspace \forall i\in\{1,2,3\}$} strain field $\epsilon$ from $I\epsilon$ for an isotropic material can be carried out in the MATLAB environment using the following procedure;

\begin{enumerate}

    \item{From a given discrete strain-sinogram, compute the three unique components of $^s\epsilon$ based on \eqref{2DInvFormula}.  Practically this can be achieved by incremental application of the \texttt{`iradon'} intrinsic function with appropriate weighting relative to the projection direction.}
    
    \item{Calculate the derivatives of the reconstructed components of ${^s}\epsilon$.  This can be implemented numerically in MATLAB by first transforming ${^s}\epsilon$ to the Fourier domain using the `\texttt{fft2}' and `\texttt{fftshift}' intrinsic MATLAB functions before multiplying by appropriate complex functions of spatial frequency and transforming back to the real domain using `\texttt{ifftshift}' and `\texttt{ifft2}'.}

    \item{From these derivatives we can compute the right hand side of Equation \eqref{ElasticModel}, which we now label $b=-\rm{Div}(C:{^s}\epsilon)$ .  Under two-dimensional plane-stress conditions, and for an isotropic material, this can be written
    \begin{align}
        \label{bx}
        b_1&=-\frac{E}{1-\nu^2}\Big(\frac{\partial{{^s}\epsilon_{11}}}{\partial x_1} + \nu \frac{\partial{{^s}\epsilon_{22}}}{\partial x_1} + (1-\nu) \frac{\partial{{^s}\epsilon_{12}}}{\partial x_2}\Big), \\
        b_2&=-\frac{E}{1-\nu^2}\Big(\nu\frac{\partial{{^s}\epsilon_{11}}}{\partial x_2} + \frac{\partial{{^s}\epsilon_{22}}}{\partial x_2} + (1-\nu) \frac{\partial{{^s}\epsilon_{12}}}{\partial x_1}\Big),
        \label{by}
    \end{align}
    where $E$ and $\nu$ are Young's modulus and Poisson's ratio respectively.  Note that similar expressions exist in the case of plane-strain.}

    \item{\label{IsepsCalc}Mask ${^s}\epsilon$ to zero outside the known boundary and then apply the LRT.  This can be implemented by successive application of the \texttt{`radon'} intrinsic function as follows;
    \[
        I{^s}\epsilon(s,\theta)=\mathcal{R}[\cos^2\theta{^s}\epsilon_{11} +2 \cos\theta\sin\theta{^s}\epsilon_{12}+\sin^2\theta{^s}\epsilon_{22}]
    \]
    where $\theta$ defines the angle of the projection direction $\xi=(\cos\theta,\sin\theta$).}

    \item{\label{residCalc}Compute the left hand side of \eqref{residual} by subtracting the result of step \ref{IsepsCalc} from the original strain-sinogram.}

    \item{\label{Ubdry}The result of step \ref{residCalc} corresponds to the right hand side of \eqref{AU=R}.  After forming $A$ through the process described in Section \ref{NumericalScheme}, solve for the unknown boundary displacements using the Moore-Penrose pseudoinverse as implemented in the \texttt{`pinv'} MATLAB intrinsic function.}

    \item{\label{PotSol}The potential part $du$ can then be recovered by solving the boundary value problem posed by \eqref{ElasticModel} together with the boundary displacements calculated in step \ref{Ubdry} applied as Dirichlet constraints.  This can be achieved using a finite element approach with the MATLAB PDE solver.}

    \item{The reconstructed strain field can then be calculated as the sum of the ${^s}\epsilon$ and $du$.}
    
\end{enumerate}

\subsection{Simulation}

The procedure outlined in Section \ref{algorithm} was implemented within MATLAB and tested on a simulated system defined on the unit disk $\mathbb{D}^2 \subset \mathbb{R}^2$.  Strain within this disk was computed from an Airy potential function of the form
\[
\phi(x_1,x_2) = x_1  \sin(3x_1 x_2^2) e^{x_1 x_2^2} + x_2 \cos(x_1^3),
\]
from which stress was defined as
\begin{equation}
\label{Airy}
\sigma = \chi\begin{bmatrix}
    \frac{\partial^2\phi}{\partial x_2^2} & -\frac{\partial^2\phi}{\partial x_1 \partial x_2}\\
    -\frac{\partial^2\phi}{\partial x_1 \partial x_2} & \frac{\partial^2\phi}{\partial x_1^2}
\end{bmatrix}
\end{equation}
where $\chi$ is the characteristic function for the disk.

Note that \eqref{Airy} guarantees $\sigma$ is divergence-free within the disk and trivially at all points outside.  The boundary of the disk is subject to a non-zero traction force calculated as $\tau=\sigma \cdot n |_{\partial\Omega}$ and shown in Figure \ref{fig:BdryTraction}.

\begin{figure}
    \centering
    \includegraphics[width=0.4\linewidth]{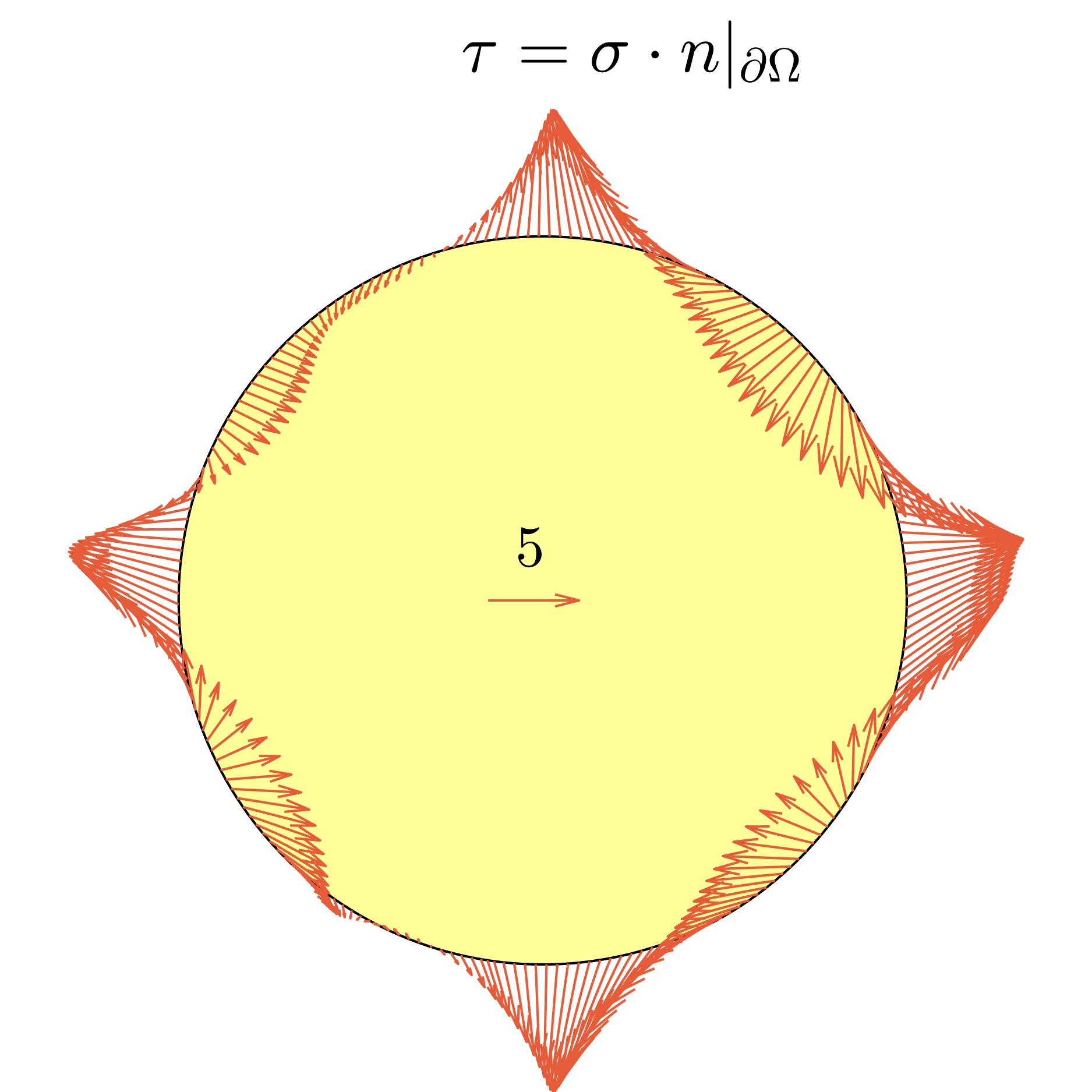}
    \caption{Boundary traction on the disk computed from the stress on $\partial \Omega$ (scale indicated by the central arrow).}
    \label{fig:BdryTraction}
\end{figure}

From Hooke's law and an assumption of plane stress in an isotropic material, strain within the disk can then be expressed as
\begin{align}
\epsilon_{11}&=\tfrac{1}{E}(\sigma_{11}-\nu\sigma_{22})\\
\epsilon_{22}&=\tfrac{1}{E}(-\nu\sigma_{11}+\sigma_{22})\\
\epsilon_{12}&=\tfrac{1}{E} (1 + \nu) \sigma_{12}.
\end{align}
Values of $E=1$ and $\nu=0.28$ were used.

The resulting strain field was forward-mapped by the LRT to provide a simulated data set for testing the algorithm. This involved computing 250 equally spaced projections over $360^\circ$ from a discretised version of the field on a regular $222 \times 222$ rectangular grid.  Projections were based on a ray spacing of 0.01 units, with the resulting sinogram shown in Figure \ref{SimResult}a.

\begin{figure}
    \begin{center}
            \includegraphics[width=0.66\linewidth]{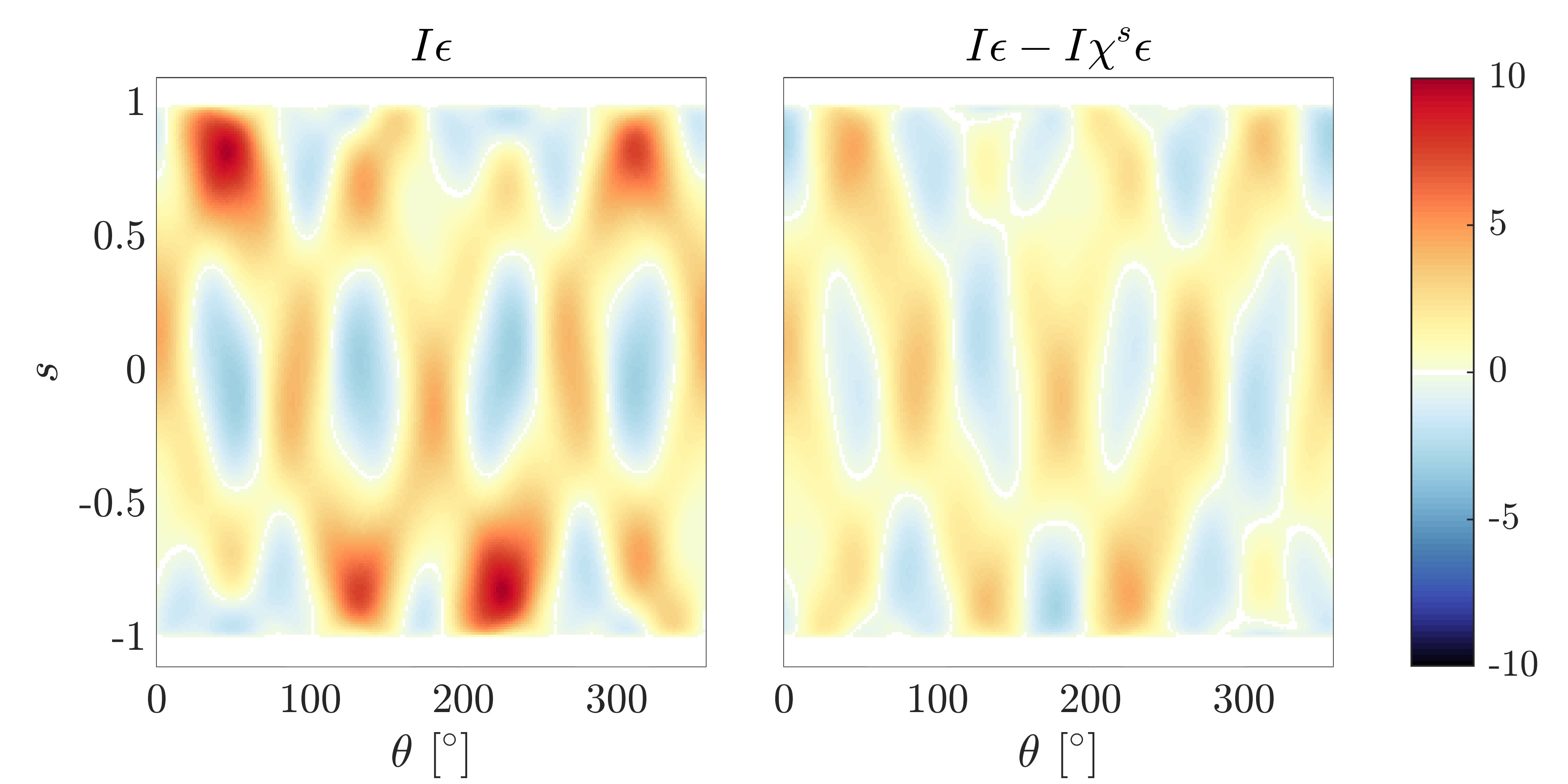}
        \includegraphics[width=0.33\linewidth]{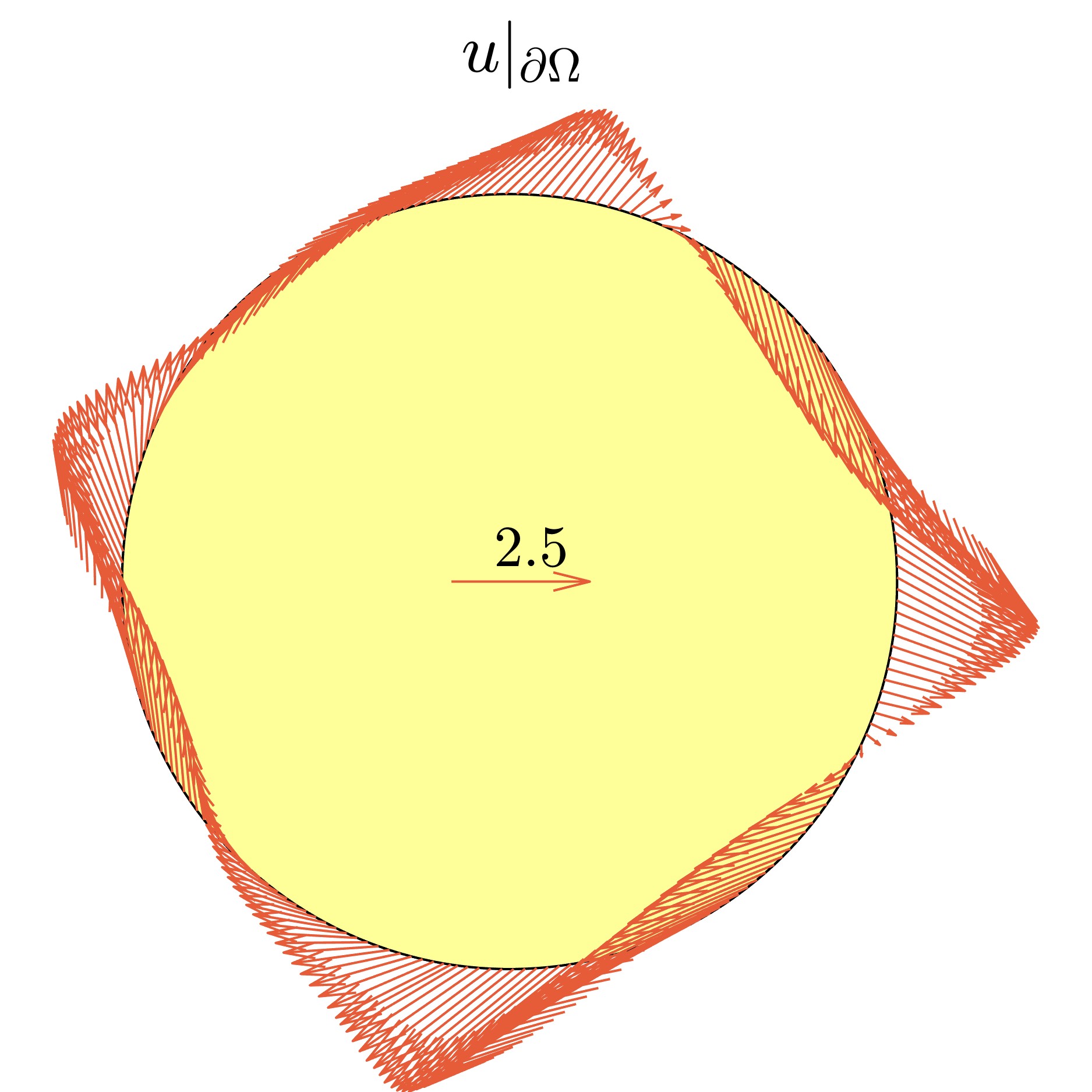}
        \put(-385,120){(a)}
        \put(-120,120){(b)}\\
        \includegraphics[width=0.24\linewidth]{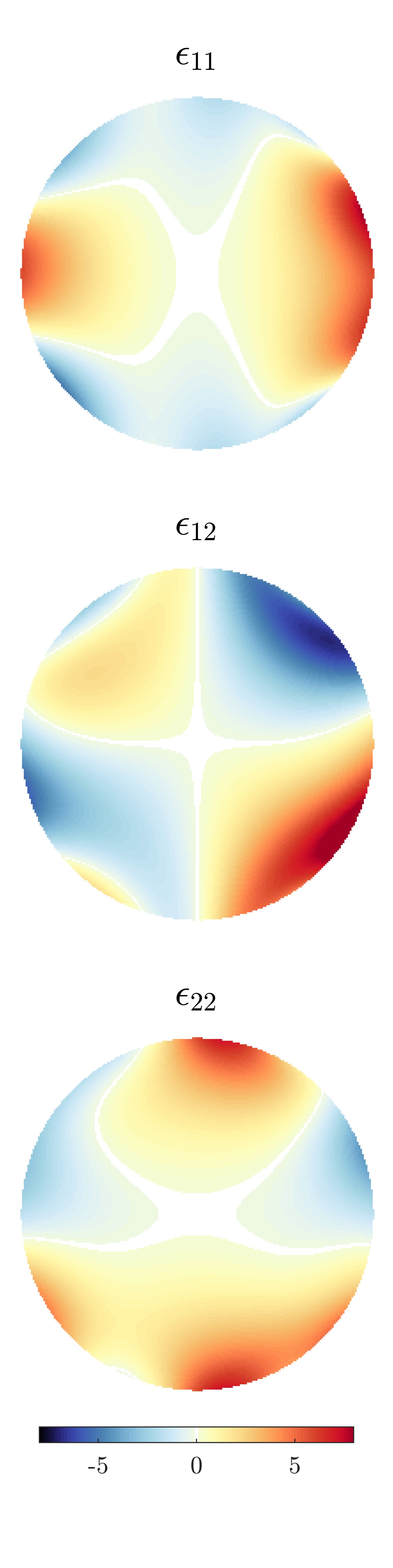}
        \includegraphics[width=0.24\linewidth]{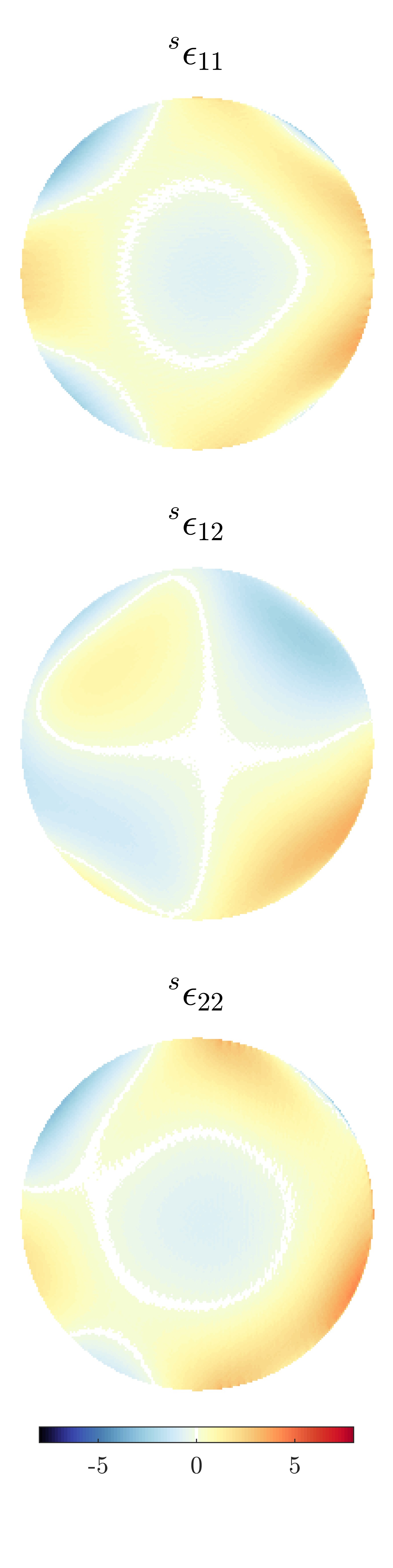}
        \includegraphics[width=0.24\linewidth]{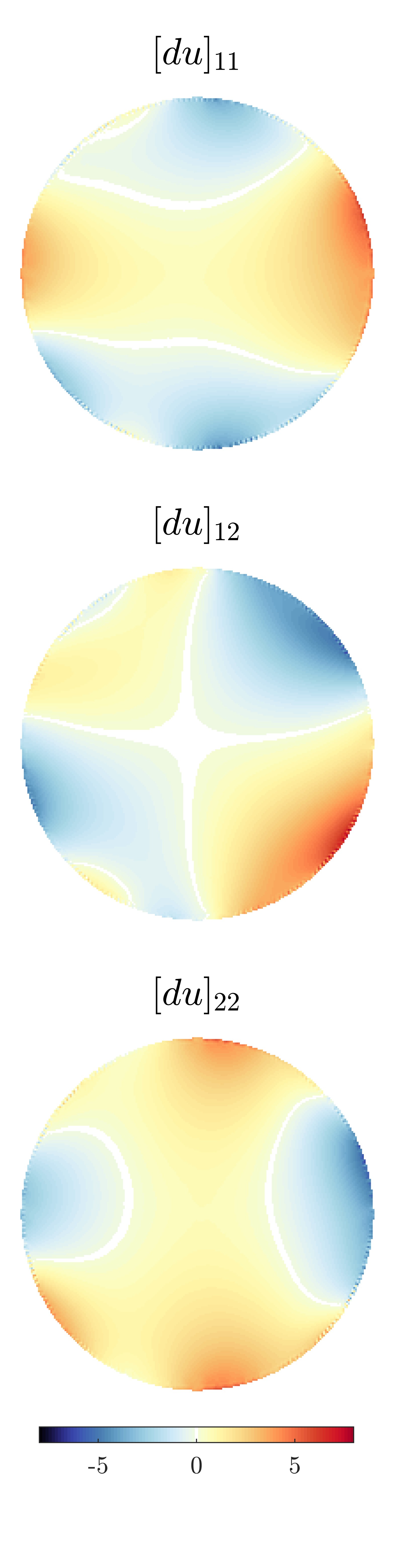}
        \includegraphics[width=0.24\linewidth]{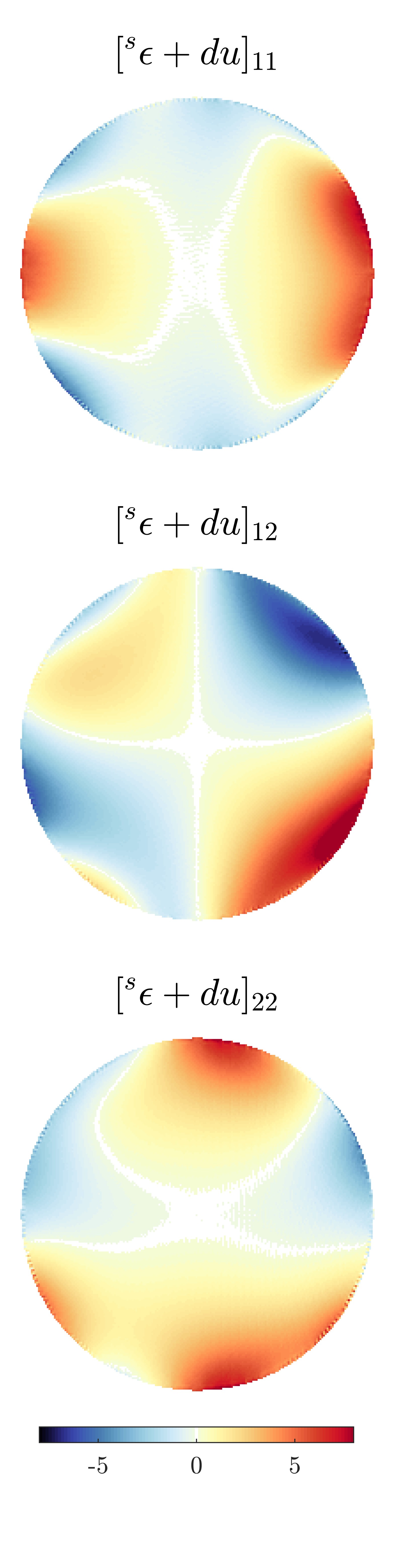}
        \put(-390,350){(c)}
        \put(-292.5,350){(d)}
        \put(-193.5,350){(e)}
        \put(-95,350){(f)}
        \caption{\label{SimResult}Simulated reconstruction; a) the original and residual LRT-sinograms, b) the boundary displacement computed from the residual LRT-sinogram (scale indicated by the central arrow), c) the true strain field, d) the solenoidal part from \eqref{2DInvFormula}, e) the potential part from step \ref{PotSol} above, and finally f) the reconstructed strain field.  }
    \end{center}
\end{figure}

Figure \ref{SimResult}c to \ref{SimResult}f shows the strain field along with its reconstruction laid out in its component parts.  The potential part (Figure \ref{SimResult}e), was computed from the boundary displacement shown in Figure \ref{SimResult}b which, in turn, was computed from the residual LRT shown in Figure \ref{SimResult}a.  The reconstruction of $du$ was carried out using a finite element model consisting of 8153 elements (typical element size of 0.03 units).

\subsection{Convergence and noise rejection}

Sensitivity to experimental error was examined by introducing zero-mean random noise to the simulated measurement above.  This noise had a standard deviation of 5\% of the average absolute value of the simulated strain-sinogram.  

Convergence of the reconstruction to the original strain field was examined as a function of the number of projections with geometry otherwise as above (i.e. same computational grid and projection spacing).  In each case, the relative error in the reconstruction was computed as an integral over the whole domain of the form
\[
    e = \frac{\int_{\mathbb{D}^2} ||^s\epsilon+du-\epsilon||_F dA}{\int_{\mathbb{D}^2} ||\epsilon||_F dA}
\]
where $||\cdot||_F$ refers to the Frobenius norm, and integrals were computed numerically as a Riemann sum of pixel values.

The resulting relative error was both a function of the number of projections and the size of the finite element mesh used to compute $du$; larger meshes require many more projections, but eventually converge to more accurate reconstructions.

\begin{figure}
    \begin{center}
        \includegraphics[width=0.8\linewidth]{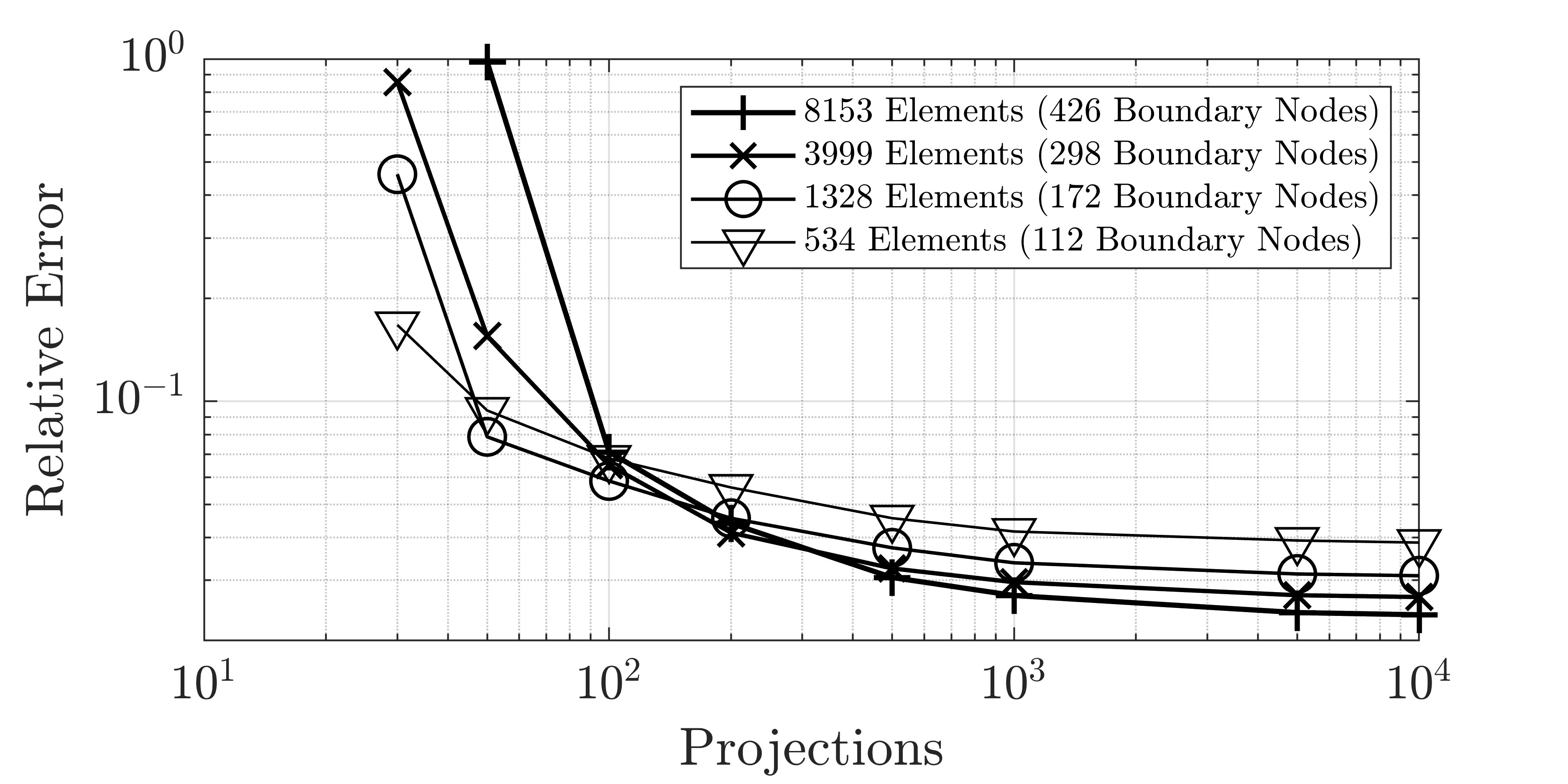}
        \caption{\label{Convergence} Convergence of reconstructions with the number of projections in the presence of random noise. Each line is for a different size finite element mesh with the corresponding number of elements and boundary nodes indicated in the legend.}
    \end{center}
\end{figure}

\section{Conclusion}

We have developed an algorithm capable of reconstructing elastic strain fields satisfying equilibrium from their LRT on connected objects whose boundary only contains one component. Our method provides full reconstruction in the energy resolved neutron transmission strain imaging problem for objects which satisfy this condition, without restrictions imposed in previous work such as no residual stress or zero boundary traction. In the case of a connected object with more than one boundary component, there will be a finite dimensional space of undetectable tensor fields which satisfy equilibrium inside the object and restrict to different rigid body motions on each component of the boundary. We will study this case in future work.

A critical step in our reconstruction is determination of the values of a potential field $d u$ on the boundary of an object from the LRT of $\chi d u$, where $\chi$ is the indicator function for the object. We showed that this determination is unique up to rigid motions on each component of the boundary, which proves our method for recovery of the elastic field gives the correct result when there is only one component of the boundary. However, the stability of this recovery is still an open problem. This stability will depend on the object, and we conjecture that the dependence will be through the curvature of the boundary. This is a topic for future work.

Figure \ref{Convergence} also provides an avenue for further investigation. Practical neutron transmission measurements provide on the order of 50-100 projections \cite{gregg2018tomographic}, and for this number Figure \ref{Convergence} shows that a relatively coarse finite element mesh may be preferable. This is consistent with the``regularization by discretization"  philosophy in inverse problems \cite{kirsch2011introduction} and work is necessary to determine the ideal mesh given a certain number of projections.

\section{Acknowledgements}

This work is supported through a number of funding sources;

Contributions by C Wensrich and T Doubikin were supported by the Australian Research Council through a Discovery Project Grant (DP170102324).  

Contributions from W Lionheart and S Holman were supported by the Engineering and Physical Sciences Research Council (EPSRC) through grant EP/V007742/1.

Contributions from A Polyakova were supported by the framework of the government assignment of the Sobolev Institute of Mathematics, project FWNF-2022-0009.

Contributions from I Svetov were supported by the Russian Scientific Foundation
(RSF), project No. 24-21-00200.

Contributions from M Courdurier were partially supported by ANID Millennium Science Initiative Program through Millennium Nucleus for Applied Control and Inverse Problems NCN19-161.

Experimental data shown in Figure \ref{fig:BE_Imaging} was measured using the RADEN instrument in the Materials and Life Sciences Institute at the Japan Proton Accelerator Research Complex (J-PARC) through Long Term Proposal 2017L0101.

The authors would also like to thank the Isaac Newton Institute for Mathematical Sciences  for support and hospitality through the program Rich and Non-linear Tomography: A Multidisciplinary Approach where work on this paper was undertaken. This program was supported by EPSRC grant number EP/R014604/1.  

While in Cambridge, all authors received support from the Simons Foundation. C Wensrich would also like to thank Clare Hall for their support and hospitality over this period.

\bibliographystyle{elsarticle-num}
\bibliography{WCbiblio}

\end{document}